\documentclass[12pt]{amsart}
\usepackage{amssymb,mathrsfs,amsmath}
\usepackage{amscd,amsmath,amsthm,amssymb,amsfonts,enumerate}

\usepackage{graphicx, color}

\def\({\left(}
\def\]{\right]}
\def\[{\left[}
\def\){\right)}

\newtheorem{thm}{Theorem}[section]
\newtheorem{prin}[thm]{Principle}
\newtheorem{asm}[thm]{Assumption}
\newtheorem{prop}[thm]{Proposition}
\newtheorem{lem}[thm]{Lemma}
\newtheorem{rem}[thm]{Remark}

\newtheorem{conj}[thm]{Conjecture}
\newtheorem{defn}[thm]{Definition}
\newtheorem{exm}[thm]{Example}

\def\P{{\mathbb P}}
\def\E{{\mathbb E}}
\def\Cov{{\mathbb Cov}}

\newcommand{\disp}{\displaystyle}
\newcommand{\bea}{$$\begin{array}{ll}}
\newcommand{\eea}{\end{array}$$}
\newcommand{\bed}{\begin{displaymath}}
\newcommand{\eed}{\end{displaymath}}
\newcommand{\ad}{&\!\!\!\disp}
\newcommand{\aad}{&\disp}
\newcommand{\barray}{\begin{array}{ll}}
\newcommand{\earray}{\end{array}}
\newcommand{\beq}[1]{\begin{equation} \label{#1}}
\newcommand{\eeq}{\end{equation}}

\newcommand{\Sb}{{\bar S}}

\newcommand{\bedd}{\bed\begin{array}{l}}
\newcommand{\eedd}{\end{array}\eed}
\newcommand{\al}{\alpha}
\newcommand{\sg}{\sigma}

\newcommand{\e}{\varepsilon}

\newcommand{\one}{{1}}

\newcommand{\wdt}{\widetilde}
\newcommand{\wdh}{\widehat}
\newcommand{\diag}{{\rm diag}}
\newcommand{\cd}{(\cdot)}
\newcommand{\rr}{\Bbb R}
\renewcommand{\L}{\mathcal {L}}
\def\one{{\hbox{1{\kern -0.35em}1}}}

\def\ka{\kappa}
\def\ph{\varphi}

\def\({\left(}
\def\]{\right]}
\def\[{\left[}
\def\){\right)}
\def\one{{\hbox{1{\kern -0.35em}1}}}

\makeatletter \@addtoreset{equation}{section}

\def\para#1{\vskip 0.4\baselineskip\noindent{\bf #1}}

\def\tr{\hbox{tr}}

\def\xh{\xi^h}

\def\ei{{\bf e_i}}

\def\ej{{\bf e_j}}
\def\argmax{\hbox{argmax}}

\begin{document}

\title{Harvesting of interacting stochastic populations }
\author[A. Hening]{Alexandru Hening }
\address{Department of Mathematics\\
Tufts University\\
Bromfield-Pearson Hall\\
503 Boston Avenue\\
Medford, MA 02155\\
United States
}
\email{alexandru.hening@tufts.edu}

\author[K. Q. Tran]{Ky Quan Tran}

\address{Department of Mathematics\\
College of Education\\
Hue University\\
Vietnam
}
\address{Department of Mathematics\\
Tufts University\\
Bromfield-Pearson Hall\\
503 Boston Avenue\\
Medford, MA 02155\\
United States
}
\email{ky.tran@tufts.edu}

	\author[T. T. Phan] {Tien Trong Phan}
\address{Department of Natural Sciences\\
Quang Binh University\\
Dong Hoi City\\
Vietnam}
\email{trongtien2000@gmail.com}

\author[G. Yin]{George Yin}
\address{Department of Mathematics\\
Wayne State University\\
Detroit, MI 48202\\
United States
}
\email{gyin@math.wayne.edu}
\maketitle

\begin{abstract}
We analyze the optimal harvesting problem for an ecosystem of species that experience environmental stochasticity. Our work generalizes the current literature significantly by taking into account non-linear interactions between species, state-dependent prices, and species injections. The key generalization is making it possible to not only harvest, but also `seed' individuals into the ecosystem. This is motivated by how fisheries and certain endangered species are controlled. The harvesting problem becomes finding the optimal harvesting-seeding strategy that maximizes the expected total income from the harvest minus the lost income from the species injections. Our analysis shows that new phenomena emerge due to the possibility of species injections.

It is well-known that multidimensional harvesting problems are very hard to tackle. We are able to make progress, by characterizing the value function as a viscosity solution of the associated Hamilton-Jacobi-Bellman (HJB) equations.  Moreover, we provide a verification theorem, which tells us that if a function has certain properties, then it will be the value function. This allows us to show heuristically, as was shown in Lungu and $\O$ksendal (Bernoulli '01), that it is almost surely never optimal to harvest or seed from more than one population at a time. 

It is usually impossible to find closed-form solutions for the optimal harvesting-seeding strategy. In order to by-pass this obstacle we approximate the continuous-time systems by Markov chains. We show that the optimal harvesting-seeding strategies of the Markov chain approximations converge to the correct optimal harvesting strategy. This is used to provide numerical approximations to the optimal harvesting-seeding strategies and is a first step towards a full understanding of the intricacies of how one should harvest and seed interacting species. In particular, we look at three examples: one species modeled by a Verhulst-Pearl diffusion, two competing species and a two-species predator-prey system.

\end{abstract}

\maketitle

\setlength{\baselineskip}{0.22in}

\tableofcontents

\section{Introduction}\label{sec:int}

Real populations never evolve in isolation. As a result, a key question in ecology is finding conditions that allow multiple species to coexist. There is a general theory for deterministic coexistence \cite{H81, H84, HJ89, HS98, ST11}. However, due to the intrinsic randomness of environmental fluctuations, deterministic models should be seen only as first order approximations of the real world. In order to get a better understanding of population dynamics we have to take into account environmental stochasticity. Recently there has been significant progress towards a general theory of stochastic coexistence \cite{SBA11, B18, BS18, HN18}.

Many species of animals live in restricted habitats and are at risk of being overharvested. Harvesting, hunting and other forms of overexploitation have already driven species to extinction. On the other hand, underharvesting can lead to the loss of valuable resources. One has to carefully balance both conservation and economic considerations in order to find the optimal harvesting strategies. It can take a population a significant amount of time to recover from large harvests. This, in combination with the random environmental fluctuations, can make it impossible for the population to survive and can lead to extinctions \cite{LES95, LES}.

In certain cases, added conservation efforts have to be made in order to save a species from extinction. Therefore, it makes sense to be able to repopulate a species by \textit{seeding} animals into the habitat. There is no reason to assume that the price of the harvesting or seeding is constant. If the harvested population is smaller the cost of harvesting is usually higher due to the fact that it is harder to find the individuals one wants to harvest. Similarly, the marginal cost of seeding will be lower, if one has a large population.
We present a model that incorporates all these factors and effects. We consider $d\geq 0$ species interacting nonlinearly in a stochastic environment where the species can be harvested as well as seeded into the system and the price of harvesting and seeding is density-dependent. The problem becomes finding the \textit{optimal harvesting-seeding strategy} that maximizes the expected total income from the harvest minus the lost income from the species seedings. Mathematically, the problem we consider
belongs to a
class of singular stochastic control problems.
 Singular stochastic control problems have been studied extensively in various settings. To mention just a few, we refer to \cite{Alvarez98, Alvarez00, Lungu, Zhu11, Hening, AH18} for single species ecosystems in random environments and \cite{Lungu01, Ky15, Ky17} for interacting populations. The reader can also find analogous results in the setting of corporate strategy \cite{Shepp96}, and optimal dividend strategies \cite{Taksar,Jin12, Scheer}. Numerical methods for optimal harvesting  have been developed in \cite{Jin12, Ky16} and capital injections have been introduced in \cite{Dickson04, Kulenko08, Scheer}.

 Considering optimal dividend problems
 in insurance and risk management
 \cite{Dickson04, Jin12, Kulenko08,  Scheer}, it was observed that higher profit can be obtained if investors are allowed  not only to remove but also to inject capital. In the harvesting setting, the idea of repopulating species (which we will call seeding) is natural and has been done for conservation efforts as well as for fisheries and agriculture. We propose a general model in which the \textit{control} consists of two components: harvesting and seeding. %Different from
In contrast to
the existing literature, in our framework,
to maximize the expected total discounted reward, the controller can add individuals of various species to maintain the system
 at a certain level and  to avoid extinction. Moreover, we work with a system of interacting species. There are few theoretical results regarding the multi-species harvesting problem \cite{Lungu01, Ky17}. In a model with several species, one needs to decide which species to harvest at a given time. In addition, our model is complicated because we also allow seeding. At a given time, there are several possibilities. One can do nothing and let the population dynamics run on its own, or one can have any possible combination of seeding and harvesting of the $d$ species.

To find the optimal harvesting-seeding strategy (also called the \textit{optimal control}) and its associated total discounted reward (also called the \textit{value function}), the usual approach is to solve the associated the Hamilton-Jacobi-Bellman (HJB) partial differential equations. However, for the singular control problems we consider, the HJB equations become a system
 of nonlinear quasi-variational inequalities.
% We deal with such problems in this paper and consider the optimal %controls.
%In this work,
We use
%in addition to
the
viscosity solution approach for partial differential equations to study the value functions and associated control problems. It is usually impossible to find closed-form solutions to the HJB system. In order to side-step this difficulty and still gain valuable information, we develop numerical algorithms to approximate the value function and the optimal harvesting-seeding strategy.
 We do this by using the Markov
 chain approximation methodology developed by Kushner and Dupuis \cite{Kushner91}.

 The main contributions of our work are the following:

\begin{enumerate}
\item We formulate the harvesting-seeding problem for a system of interacting species living in a stochastic environment.
\item We establish the finiteness and the continuity of the value function and characterize
 the value function as a viscosity solution of an associated HJB
 system of quasi-variational inequalities.
\item We develop numerical approximation schemes based on the Markov chain approximation method.
\item We discover new phenomena by analyzing natural examples for one and two-species systems.
\end{enumerate}

 The rest of our work is organized as follows. In Section \ref{sec:for} we describe our model and the main results. Particular examples are explored using the newly developed numerical schemes in Section \ref{sec:fur}.  Finally, all the technical proofs appear in the appendices.

\section{Model and Results}\label{sec:for}
Assume we have a probability space $(\Omega, \mathcal{F}, \P)$ satisfying the usual conditions. We consider $d$ species interacting nonlinearly in a stochastic environment. We model the dynamics as follows. Let $\xi_i(t)$ be the
density of the $i$th species at time $t\geq 0$, and
denote by $\xi(t)=(\xi_1(t), \dots, \xi_d(t))'\in \rr^d$ (where $z'$ denotes
the transpose of $z$) the column vector recording all the species densities.

One way of adding environmental stochasticity to a deterministic system is based on the assumption that the environment mainly affects the growth/death rates of the populations. This way, the growth/death rates in an ODE (ordinary differential equation) model are replaced by their average values to which one adds a white noise fluctuation term; see \cite{T77, B02, G88, ERSS13, SBA11, G84} for more details.

Under this assumption the dynamics becomes
\beq{e.1} d \xi(t)=b(\xi(t))
dt+\sigma(\xi(t)) d w(t).\eeq
where $w\cd=\(w_1\cd, ..., w_d \cd\)'$ is a
$d$-dimensional standard Brownian motion and $b,\sigma:[0,\infty)^d\to[0,\infty)^d$ are smooth enough functions.  Let $S=(0, \infty)^d$ and $\Sb =[0, \infty)^d$.  We assume that $b(0)=\sg(0)=0$ so that $0$ is an equilibrium point of \eqref{e.1}. This makes sense because if our populations go extinct, they should not be able to get resurrected without external intervention (like a repopulation/seeding event). If $\xi_i(t_0)=0$ for some $t_0\ge 0$, then $\xi_i(t)=0$ for any $t\ge t_0$. Thus, $\xi(t)\in \Sb$ for any $t\ge 0$.

For $x, y\in \rr^d$, with $x=(x_1, \dots, x_d)'$ and $y=(y_1, \dots, y_d)'$, we write $x\le y$ or $y\ge x$ if $x_j\le y_j$ for each $j=1, \dots, d$, while $x < y$ if $x_j < y_j$ for each $j=1, \dots, d$..
We also define the scalar product $x\cdot y=\sum_{j=1}^d x_j y_j$. For a real number $a$, we denote $a^+=\max\{a, 0\}$ and $a^-=\max\{-a, 0\}$. Thus, $a=a^+-a^-$ and $|a|=a^++a^-$. For $x=(x_1, \dots, x_d)\in \rr^d$, $x^+=\big( x_1^+, \dots, x_d^+\big)'$ and $x^-=\big( x_1^-, \dots, x_d^-\big)'$. Let $\ei\in \rr^d$ denote the unit vector in the $i$th direction for $i=1, \dots, d$.

 To proceed, we introduce the generator of the process $\xi(t)$.
For a twice continuously differentiable function $\Phi(\cdot): \rr^d\mapsto \rr$, we define
\bea
\L \Phi(x)\ad=b(x)\nabla \Phi(x)+\dfrac{1}{2}\tr\big(\sigma(x)\sigma'(x)\nabla^2 \Phi(x)\big),\eea
where $\nabla \Phi(\cdot)$ and $\nabla^2 \Phi(\cdot)$ denote the gradient and Hessian matrix of $\Phi(\cdot)$, respectively.

Next, we have to add harvesting and seeding to \eqref{e.1}.  Let $Y_i(t)$ denote the amount of species $i$ that has been harvested up to time $t$ and set
$Y(t)=(Y_1(t), \dots, Y_d(t))'\in \rr^d$.
Let $Z_i(t)$ denote the amount of species $i$ seeded into the system up to time $t$ and set
$Z(t)=(Z_1(t), \dots, Z_d(t))'\in \rr^d$.
The dynamics of the $d$ species that takes into account harvesting and seeding is given by
\beq{e.2}X(t)=x+\int\limits_0^t b(X(s)) ds +  \int\limits_0^t  \sigma(X(s)) dw(s) - Y(t) +Z(t),\eeq
where  $X(t)=(X_1(t), \dots, X_d(t))'\in \rr^d$ are the species densities at time $t\geq 0$. We also assume the initial species densities are
\beq{e.3}X(0-)=x\in \Sb .\eeq
\para{Notation.} For each time $t$,
$X(t-)$ represents the state
before harvesting starts at time $t$, while $X(t)$ is the state immediately after.
Hence $X(0)$ may not be equal to $X(0-)$ due to an instantaneous harvest $Y(0)$ or an instantaneous
seeding $Z(0)$
at time $0$. Throughout the paper we use the convention that $Y(0-)=Z(0-)=0$.
The jump sizes of $Y(t)$ and $Z(t)$ are denoted by $\Delta Y(t):=Y(t) - Y(t-)$ and  $\Delta Z(t):=Z(t) - Z(t-)$, respectively. We use $Y^c(t):=
Y(t)  - \sum\limits_{0\le s\le t} \Delta Y(s)$ and
 $Z^c(t):=
Z(t)  - \sum\limits_{0\le s\le t} \Delta Z(s)$ to denote the continuous part of $Y$ and $Z$. Also note that $\Delta X(t):=
X(t) -  X(t-)=  \Delta Z(t)-\Delta Y(t)$ for any $t \ge 0$.

Let $f_i: \Sb\mapsto (0, \infty)$ represent the instantaneous marginal yields accrued
from exerting the harvesting strategy $Y_i$ for the species $i$, also known as the price of species $i$. Let $g_i: \Sb \mapsto (0, \infty)$ represent the total cost we need to pay for the seeding strategy $Z_i$ on species $i$. We will set $f=(f_1, \dots, f_d)'$ and $g=(g_1, \dots, g_d)'$.
For a harvesting-seeding strategy $(Y, Z)$ we define the \textit{performance function} as
\beq{e.4}
J(x, Y, Z):= \E_{x}\bigg[\int\limits_0^{\infty} e^{-\delta s} f\( X(s-) \)\cdot dY(s)-\int\limits_0^{\infty} e^{-\delta s} g\( X(s-)\) \cdot dZ(s)\bigg],
\eeq
where 
 $\delta> 0$ is the discounting factor, $\E_{x}$ denotes the expectation with respect
to the probability law when the initial densities are $X(0-)=x$, and $f(X(s-))\cdot dY(s):=\sum_{i=1}^n f_i(X(s-)) dY_i(s)$.

\para{Control strategy.}
Let $\mathcal{A}_{x}$ denote the collection of all admissible controls with initial condition $x$. A harvesting-seeding strategy $(Y, Z)$ will be in $\mathcal{A}_x$ if it satisfies
 the following
conditions:
\begin{itemize}
\item[{\rm (a)}] the processes $Y(t)$ and $Z(t)$ are right continuous, nonnegative, and nondecreasing with respect to $t$,
\item[{\rm (b)}] the processes $Y(t)$ and $Z(t)$ are adapted to
$\sg\{w(s): 0\le s\le t\}$, augmented by the $\P$-null sets,
\item[{\rm (c)}] The system \eqref{e.2} has a unique solution $X\cd$ with $X(t)\ge 0$ for any $t\ge 0$.
\end{itemize}

%Let $\mathcal{A}_{x}$ denote the collection of all admissible %control strategies with initial condition $X(0-)=x$.
\textbf{The optimal harvesting-seeding problem.} The problem we will be interested in is to maximize the
performance function and find an optimal harvesting strategy $(Y^*, Z^*)\in \mathcal{A}_{x}$ such that
\beq{e.5}
J(x, Y^*, Z^*)=V(x):= \sup\limits_{(Y, Z)\in \mathcal{A}_{x}}J(x, Y, Z).
\eeq
The function $V(\cdot)$ is called the \textit{value function}.
\begin{rem}
We note that the optimal harvesting strategy might not exist, i.e. the maximum over $\mathcal{A}_x$ might not be achieved in $\mathcal{A}_x$.
\end{rem}
\begin{asm}\label{a:1}
We will make the following standing assumptions throughout the paper.
\begin{itemize}
\item[{\rm (a)}]  The functions $b(\cdot)$ and $\sigma(\cdot)$ are continuous. Moreover, for any initial condition $x\in \Sb $, the uncontrolled system \eqref{e.1}  has a unique global solution.
\item[{\rm (b)}]
For any $i=1, \dots, d$, $x, y\in \rr^d$,
${f_i(x)}< g_i(y)
$; $f_i(\cdot)$, $g_i(\cdot)$ are continuous and non-increasing functions.
\end{itemize}
\end{asm}
\begin{rem}
 Note that Assumption \ref{a:1} (a) does not put significant restraints on the dynamics of the species. Our framework therefore contains a very broad class of models. In particular, this covers all Lotka-Volterra competition and predator-prey models as well as the more general Kolmogorov systems \cite{Dang, Li09, Mao2006, HN18}.
The continuity and monotonicity of the functions $f\cd, g\cd$ from Assumption \ref{a:1}(b) are standard \cite{Alvarez00, Zhu11, Ky17}. The additional requirement that ${f_i(x)}< g_i(y)
$ for any $x, y\in \Sb $ can be explained as follows: the cost of seeding an amount of a species is always higher than the benefit received from harvesting the same amount. This makes sense because in order to seed the species, one has to have access to a pool of individuals of this species. For this, one either has to keep individuals at a specific location (thus using resources to sustain them) or one has transport/buy indidivuals. In the setting of optimal dividend payments, these extra costs reflect penalizing factors \cite{Kulenko08} and transaction costs \cite{Jin12, Scheer}.
\end{rem}
We collect some of the results we are able to prove about the value function.

\begin{prop} \label{prop:2} Let Assumption \ref{a:1} be satisfied. Then the following assertions hold.
	\begin{itemize}
	\item[\rm (a)] For any $x, y\in \Sb $,
	\beq{e.3.5}V(y)\le V(x) - f(x)\cdot (x-y)^++g(x)\cdot (x-y)^-.\eeq
	
	\item[\rm (b)] If $V(0)<\infty$, then $V(x)<\infty$ for any $x\in \Sb $ and $V(\cdot)$ is Lipschitz continuous.
	\end{itemize}
\end{prop}

	\begin{exm}
	{\rm
		In the current setting, contrary to the regular harvesting setting without seeding, $V(0)$ can be nonzero because of the benefits from seedings. Consider the single species system given by
		\beq{e.3.7}
		dX(t)=X(t)(a-bX(t))dt + \sg X(t) dw(t)-dY(t) +dZ(t), \quad X(0)=x,
		\eeq
		where $a, b,$ and $\sg$
		are constants and the price function is $f(x)=1, x\ge 0$.
		It has been shown in \cite{Alvarez98} that, if there is no seeding, the value function is given by
		$$V_0(x)= \psi(x) \text{ for }x<x^*, \quad V_0(x)= x-x^*+\psi(x^*) \text{ for }x\ge x^*,$$
		where $x^*\in (0, \infty)$  and $\psi: [0, \infty)\to [0, \infty)$ is twice continuously differentiable, and $\psi
		(x)>x$ for all $x\in (0, x^*]$.
		Let $g(x)=\kappa\in \mathbb{R}$, where $1<\kappa<\psi(x^*)/x^*$. Let $(Y, Z)\in \mathcal{A}_0$ be such that  $J(0, Y, 0)=V_0(x)$ and $Z(t)=Z(0)=x^*$ for all $t\ge 0$. Then
		$$V(0)\ge J(0, Y,Z)\ge \psi(x^*)-\ka x^*>0.$$
		Since $V(0)>0$, the system does not get depleted in a finite time under an optimal harvesting strategy.
	}	
\end{exm}

\begin{prop}
	Let Assumption \ref{a:1} be satisfied.  Moreover, suppose that there is a positive constant $C$ such that \begin{equation}b_i(x)\le \delta x_i+C, \quad x\in\Sb , \quad i=1, \dots, d.\end{equation}
	Then there exist a positive constant $M$ such that
	$$V(x)\le \sum\limits_{i=1}^d f_i(0) x_i+ M, \quad x\in\Sb .$$
\end{prop}

\begin{rem}
	{\rm
		We note
that the condition  on the drift $b\cd$
		is very natural.
		Consider the one-dimensional dynamics given by
			\beq{e.3.16a}
		dX(t)=bX(t)dt + \sg X(t) dw(t)-dY(t) +dZ(t), \quad X(0)=x,
		\eeq
		 with  $f(x)=1, x\ge 0$ and any function $g\cd$.
		It is clear that if $b>\delta$,  the value function in the harvesting problem with no seeding is $$V_0(x)=\inf\limits_{(Y, Z)\in \mathcal{A}_{x}, Z=0 }J(x, Y, Z)=\infty \quad \text{for all}\quad x>0.$$
As a result the value function for \eqref{e.3.16a} will be $V(x)=\infty, x>0.$

Seeding
can also change the finiteness of the value function. Indeed,
consider the harvesting problem
	\beq{e.3.16b}
dX(t)=b(X(t))dt + \sg (X(t)) dw(t)-dY(t) +dZ(t), \quad X(0)=x,
\eeq
with  $f(x)=1, g(x)=2, x\ge 0$.	Suppose that $g(x)=\sg(x)=0$ for $x<1$ and $b(x)=(1+\delta) x( 1-x)$ and $\sg(x)=0$ for $x>1$. Then it is clear that without seeding we get the value function $V_0(x)=\infty$ for $x>1$ and $V_0(x)=x$ for $x\le 1$. When seeding is allowed, we have $V(x)=\infty$ for all $x\ge 0$.
	}
\end{rem}
	
We get the following characterization of the value function.
\begin{thm} \label{t:value}
	Let Assumption \ref{a:1} be satisfied and suppose $V(x)<\infty$ for $x\in {S}$.
	The value function
	$V\cd$ is a viscosity solution to the HJB equation
\begin{equation}\label{e:HJB}
\max\limits_{i} \bigg\{(\L-\delta)V(x), f_i(x)-\dfrac{\partial V}{\partial x_i}(x),  \dfrac{\partial V}{\partial x_i}(x)-g_i(x)\bigg\}= 0.
\end{equation}
	\end{thm}

\begin{rem}
Theorem \ref{t:value} is a theorem that tells us how to find the value function. The problem is that the solutions of \eqref{e:HJB} are not always smooth enough for $\L V$ to make sense. This is why we work with viscosity solutions of \eqref{e:HJB}.

We next explain what a viscosity solution means. For
	any $x^0\in S$ and any function $\phi\in C^2(S)$ such that $V(x_0)=\phi(x_0)$ and $V(x)\geq \phi(x)$
	for all $x$ in a neighborhood of $x^0$, we have
\begin{equation*}	
  \max\limits_{i} \bigg\{(\L-\delta)\phi(x^0), f_i(x^0)-\dfrac{\partial \phi}{\partial x_i}(x^0),  \dfrac{\partial \phi}{\partial x_i}(x^0)-g_i(x^0)\bigg\}\le 0.
\end{equation*}
Similarly, for any $x^0\in S$ and any function $\ph\in C^2(S)$ satisfying $V(x_0)=\phi(x_0)$ and $V(x)\leq \phi(x)$
	for all $x$ in a neighborhood of $x^0$, we have
\begin{equation*}\max\limits_{i} \bigg\{(\L-\delta)\ph(x^0), f_i(x^0)-\dfrac{\partial \ph}{\partial x_i}(x^0),  \dfrac{\partial \ph}{\partial x_i}(x^0)-g_i(x^0)\bigg\}\ge 0.
\end{equation*}
This extends the results from \cite{HS55a, HS55b, Lungu01} where the authors had to assume that the coefficients $b, \sigma$ are bounded or the prices $f_i$ are not density-dependent. Usually the functions $b,\sigma$ are not bounded and the prices depend on the densities of the species. Moreover, we consider both harvesting and seeding. Therefore, our results provide a significant generalization of those from \cite{HS55a, Lungu01}.
\end{rem}

We also get the following verification t heorem, that tells us that if a function satisfies certain properties, then it will be the value function. We note that this is natural analogue with seeding of Theorem 2.1 from \cite{Lungu01,ALO16}.

\begin{thm}\label{t:ver}
	Let Assumption \ref{a:1} be satisfied. Suppose that there exists a function $\Phi : \Sb  \mapsto [0, \infty)$ such that
	$\Phi\in C^2(\Sb )$   and that $\Phi\cd$ solves the following coupled system of quasi-variational inequalities
	\beq{e:ineq}
	\sup\limits_{(x, i)} \bigg\{(\L-\delta)\Phi(x), f_i(x)-\dfrac{\partial \Phi}{\partial x_i}(x),  \dfrac{\partial \Phi}{\partial x_i}(x)-g_i(x)\bigg\}\le 0,
	\eeq
	where $(\L-\delta)\Phi(x)=\L \Phi(x)-\delta \Phi(x)$. Then the following assertions hold.
	\begin{itemize}
		\item[{\rm (a)}]  We have
\begin{equation}\label{e:mon}		
V(x)\le \Phi(x)\quad \text{for any } x\in {S}.
\end{equation}
		\item[{\rm (b)}]   Define the non-intervention region 
		$$\mathcal{C}= \left\{x\in S:  f_i(x)<\dfrac{\partial \Phi}{\partial x_i}(x)<g_i(x) \right\}.$$
		Suppose that 
\begin{equation}\label{e:hjb_1}
(\L -r)\Phi(x)=0,
\end{equation}
 for all $x\in \mathcal{C}$, and that there exists a harvesting-seeding strategy $\big(\wdt{Y}, \wdt{Z}\big)
		\in \mathcal{A}_{x}$
		and a corresponding process $\wdt{X}$
		such that the following statements hold.
		
		\medskip
		
		\begin{itemize}
			\item[{\rm (i)}] $\wdt{X}(t)\in \overline{\mathcal{C}}$ for Lebesgue almost all $t\ge 0.$
			\item[{\rm (ii)}] $\int\limits_0^{t} \left[\nabla \Phi(\wdt{X}(s))- f(\wdt{X}(s))\right] \cdot d\wdt{Y}^c(s)=0$ for any $t\ge
			0$.
			\item[{\rm (iii)}] $\int\limits_0^{t} \Big[g(\wdt{X}(s))-\nabla \Phi(\wdt{X}(s)) \Big] \cdot d\wdt{Z}^c(s)=0$ for any $t\ge
			0$.
		
			\item[{\rm (iv)}] If $\wdt{X}(s)\ne \wdt{X}(s-)$, then
			$$\Phi(\wdt{X}(s))-\Phi(\wdt{X}(s-))=- f(\wdt{X}(s-)) \cdot \Delta \wdt{Z}(s)
			.$$
	\item[{\rm (v)}] $\lim\limits_{N\to\infty}E_{x}\Big[e^{-rT_N}\Phi(\wdt{X}(T_N))\Big]=0$, where for each $N=1, 2, \dots$,
			\beq{kyy}\beta_N:=\inf\{t\ge 0: |X(t)|\ge N\},\quad T_N: = N \wedge \beta_N.\eeq
		\end{itemize}
		Then $V(x)=W(x)$ for all $x\in S$, and $\big(\wdt{Y},\wdt{Z}\big)$ is an optimal harvesting-seeding strategy.
	\end{itemize}
\end{thm}
\begin{rem}
Following \cite{Lungu01} we note that if we can find a function satisfying \eqref{e:ineq}, \eqref{e:mon} and \eqref{e:hjb_1}, then one can construct a strategy satisfying assumptions (i), (ii), (iii) and (iv) from Theorem \ref{t:ver} part b) by solving the Skorokhod stochastic differential equation for the reflection of the process $X(t)$ in the domain $\mathcal{C}$. We refer the reader to \cite{Lungu01, B98, F16, LS84} for more details about Skorokhod stochastic differential equations.
\end{rem}

We can extend Principle 2.1 from \cite{Lungu01} as follows.
\begin{prin}[\textbf{One-at-a-time principle}]
Suppose the diffusion matrix $\sg(x)\sg'(x)$ is nondegenerate for all $x\in S$. Then it is almost always optimal 
to harvest or to seed from at most one species at a time.
\end{prin}
\begin{proof}
We follow \cite{Lungu01}. Assume for simplicity $d=2$ so that we have two species. The non-intervention region $\mathcal{C}$ is bounded by the four curves curves $\Lambda^f_1, \Lambda^f_2, \Lambda^g_1, \Lambda^g_2$ given by
$$
\Lambda^f_i = \left\{(x_1,x_2)\in S~\Bigg|~ \frac{\partial \Phi}{\partial x_i}=f_i(x_1,x_2)\right\}
$$
and
$$
\Lambda^g_i = \left\{(x_1,x_2)\in S~\Bigg|~ \frac{\partial \Phi}{\partial x_i}=g_i(x_1,x_2)\right\}.
$$

Note that we would have simultaneous harvesting and seeding of species $i$ only when the process is at $\Lambda^f_i \cap \Lambda^g_i $, simultaneous harvesting of the two species only when the process is at $\Lambda^f_1 \cap \Lambda^f_2 $, simultaneous harvesting of species 1 and seeding of species $2$ only when the process is at $\Lambda^f_1 \cap \Lambda^g_2$, etc.  Now, if the diffusion is non-degenerate, the probability it hits a set of the form $\Lambda^f_i \cup \Lambda^f_j$ for $i\neq j$ or $\Lambda^f_i \cup \Lambda^g_j $ is equal to zero. This argument can be extended to $n$ dimensions - see Principle 2.1 from \cite{Lungu01}.
\end{proof}

\subsection{Numerical Scheme}\label{sec:num}
A closed-form solution to the HJB equation from Theorem \ref{t:value} is virtually impossible to obtain. Moreover, the initial value of $V(0)$ is not specified. In order to by-pass these difficulties and to gain information about the value function and the optimal harvesting-seeding strategy we provide a numerical approach.
Using the Markov chain approximation method \cite{Budhiraja07, Jin12, Kushner91, Kushner92},
we
construct a controlled  Markov chain
in discrete time to approximate the controlled diffusions.
Let $h>0$ be a discretization parameter.
Since the real population densities cannot be infinite, we choose a large number $U>0$ and define the class $\mathcal{A}^U_{x}\subset \mathcal{A}_{x}$ that consists of strategies $(Y,Z)\in \mathcal{A}_x$ such that the resulting process $X$ stays in $[0,U]^d$ for all times. The class $\mathcal{A}^U_{x}$ can be constructed by using Skorokhod stochastic differential equations \cite{B98, F16, LS84} in order to make sure that the process stays in $[0,U]^d$ for all $t>0$.

Let $(\tilde Y^U, \tilde Z^U)\in \mathcal{A}^U_{x}$ and $V^U(x)$ be defined as the optimal harvesting-seeding strategy and the value function when we restrict the problem to the class $\mathcal{A}^U_{x}\subset \mathcal{A}_{x}$
\begin{equation}\label{e:VU}
J(x, \tilde Y^U, \tilde Z^U)=V^U(x):=\sup\limits_{(Y, Z)\in \mathcal{A}_{x}^U} J(x, Y, Z)
\end{equation}

\begin{rem}

We conjecture that, generically, the optimal strategy will live in $\mathcal{A}_x^U$ for $U$ large enough, i.e. there exists $U>0$ such that for all $x\in [0,U]^d$ we have
$$
J(x, Y^*, Z^*)=V(x):= \sup\limits_{(Y, Z)\in \mathcal{A}_{x}}J(x, Y, Z) =V^U(x):= \sup\limits_{(Y, Z)\in \mathcal{A}_{x}^U} J(x, Y, Z)=J(x, \tilde Y^U, \tilde Z^U).
$$
The verification Theorem \ref{t:ver} provides a heuristic argument for this conjecture.

\end{rem}
Assume without loss of generality that $U$
is an integer multiple of $h$.
Define
$$S_{h}: = \{x=(k_1 h, \dots, k_d h)'\in \rr^d: k_i=0, 1,  2, \dots\}\cap [0, U]^d.$$
Let $\{X^h_n: n=0, 1, \dots\}$
be a discrete-time controlled Markov chain with state space $S_{h}$. We define the difference
$$\Delta X_n^h = X_{n+1}^h-X_{n}^h.$$
At any discrete-time  step $n$, one can either harvest, seed, or do nothing. We use $\pi^h_n$ to denote the action at step $n$, where $\pi^h_n=-i$ if there is seeding of species $i$, $\pi^h_n=0$ if there is no seeding or harvesting of species $i$, and $\pi^h_n=i$ if there is harvesting. Denote by $\Delta Y^h_n$ and $\Delta Z^h_n$ the harvesting amount and the seeding amount for the chain at step $n$, respectively.
If $\pi^h_n=0$, then the increment
$\Delta X_n^h$ is to behave like an increment of $\int b dt +\int \sg dw$ over a small time interval. Such  a step is also called ``diffusion step''. If $\pi^h_n=-i$, then $\Delta Y^h_n=0\in \rr^d$ and $\Delta Z^h_n=h\ei$. If $\pi^h_n=i$, then $\Delta Y^h_n=h\ei$ and $\Delta Z^h_n=0\in \rr^d$. Note that $\Delta X^h_n=-\Delta Y^h_n + \Delta Z^h_n$.
Moreover, we can write
\beq{e.4.1}
\Delta X_n^h = \Delta X_n^h I_{\{ \text{diffusion step at } n\}} +  \Delta X_n^h I_{\{\text{harvesting step at }n\}}+ \Delta X_n^h I_{\{\text{seeding step at }n\}}.
\eeq
For definiteness, if $X^{h}_{n, i}$ is the $i$th component of the vector $X^h_n$ and $\{j: X^{h}_{n, j}=U\}$ is non-empty, then step $n$ is a harvesting step on species $\min\{j: X_{n, j}^{h}=U\}$.
Let $\pi^h = (\pi_0^h, \pi_1^h, \dots)$ denote the sequence of control actions.
We denote by $p^h\(x, y)|\pi\)$ the transition probability from state $x$ to another state $y$ under the control $\pi$.
Denote
$\mathcal{F}^h_n=\sigma\{X^h_m, \pi^h_m, m\le n\}$.

The sequence $\pi^h$
is said to be admissible if it satisfies the following conditions:
\begin{itemize}
	\item[{\rm (a)}]
	$\pi^h_n$ is
	$\sigma\{X^h_0, \dots, X^h_{n}, \pi^h_0, ..., \pi^h_{n-1}\}-\text{adapted},$
	\item[{\rm (b)}]  For any $x\in S_h$, we have
	$$\P\{ X^h_{n+1} =x | \mathcal{F}^h_n\}= \P\{X^h_{n+1}= x | X^h_n, \pi^h_n\} = p^h( X^h_n, x| \pi^h_n),$$
	\item[{\rm (c)}] Denote by $X^{h}_{n, i}$ the $i$th component of the vector $X^h_n$. Then \beq{}
	\barray
	\aad \P\big( \pi^h_{n}=\min\{j: X^{h}_{n, j} = U\}  | X^{h}_{n, j} = U \text{ for some } j\in \{1, \dots, d \}, \mathcal{F}^h_n\big)=1.
	\earray\eeq
	\item[{\rm (d)}] $X^h_n\in S_h$ for all $n=0, 1, 2, \dots$.
\end{itemize}
The class of all admissible control sequences $\pi^h$ for initial state $x$ will be denoted by
$\mathcal{A}^h_{x}$.

For each
couple
%[??triple?? Ky's reply: I have changed to ``couple''. Thanks]
 $(x, i)\in S_h\times \{0, \pm 1, \dots, \pm d\}$,
we define
a family of the interpolation intervals $\Delta t^h (x, i)$. The values of $\Delta t^h (x, i)$ will be specified later. Then we define
\beq{e.4.3}
\barray
\aad t^h_0 = 0,\quad  \Delta t^h_m = \Delta t^h(X^h_m, \pi^h_m),
\quad  t^h_n = \sum\limits_{m=0}^{n-1} \Delta t^h_m.\\
\earray
\eeq
For $x\in S_h$ and $\pi^h\in \mathcal{A}^h_{x}$, the performance function for the controlled Markov chain is defined as
\beq{e.4.4}
J^h(x, \pi^h) =  \E\sum_{m=1}^{\infty} e^{-\delta t_m^h}\bigg[f (X^h_m)\cdot \Delta Y_{m}^{h}-g(X^h_m) \cdot \Delta Z_{m}^{h}   \bigg].
\eeq
The value function of the controlled Markov chain is
\beq{e.4.5}
V^h(x) = \sup\limits_{\pi^h\in \mathcal{A}^h_{x}} J^h (x, \pi^h).
\eeq

\begin{thm}
	Suppose Assumptions \ref{a:1} and \ref{a:2} hold. Then $V^h(x)\to V^U(x)$ as $h\to 0$.  
	Thus, for sufficiently small $h$, a near-optimal harvesting-seeding strategy of the controlled Markov chain is also a near-optimal harvesting-seeding policy of the original continuous-time problem.
\end{thm}

\section{Numerical Examples}\label{sec:fur}
\subsection{Single species system.}
We consider a single
species ecosystem. The dynamics that includes harvesting and seeding will be given by
\beq{e.6.1}
d X(t) = X(t)\big(b-cX(t)\big)
 + \sg X(t) dw(t)-dY(t) +dZ(t).
\eeq
For an admissible strategy $(Y, Z)$ we have \beq{e.6.2}J(x, Y, Z)=\E\left[\int_0^\infty e^{-\delta s} f(X(s-))dY(s)-\int_0^\infty e^{-\delta s} g(X(s-))dZ(s)\right].\eeq
Based on the algorithm constructed above and in Appendix \ref{sec:alg}, we carry out the computation by value iterations.
Let $(Y_0,Z_0)$ be the policy that drives the system to extinction
immediately and has no seeding. Then
$J(x, Y_0, Z_0)=f(x)x$ for all $x$.
Recall from \cite{Alvarez98} that $J(x, Y_0, Z_0)$
is also referred to as current harvesting potential.
Letting $(Y_0, Z_0)$ be the initial strategy, we set the initial values $$V_0^h(x)=f(x)x, \quad x = 0, h, 2h, \dots, U = 10.$$
We  outline how to find the values of $V(\cdot)$ as follows.
At each level $x=h,2h, \dots, U$, denote by $\pi(x, n)$ the action one chooses, where $\pi(x, n)=1$ if there is harvesting, $\pi(x, n)=-1$ if there is seeding, and
$\pi(x, n)=0$ if there is no harvesting or seeding. We initially let $\pi(x, 0)=1$ for all $x$ and we try to find better harvesting-seeding strategies.
We find an improved value $V^h_{n+1}(x)$
and record the corresponding optimal action by
$$\pi(x, n)=\argmax \left\{i=-1, 0, 1: V^{h, i}_{n+1} (x, \al)\right\} ,\quad
V^h_{n+1}(x)= V^{h, \pi(x, n)}_{n+1} (x),$$
where
\bea
\aad V^{h, 1}_{n+1}(x)=V^h_n(x-h) +  f(x)h,\\
\aad V^{h, -1}_{n+1}(x)=V^h_n(x+h) - g
(x)h,\\
\aad V^{h}_{n+1, 0} (x) =e^{-\delta \Delta t^h(x, 0) }\Big[ V^{h}_{n} (x+h) p^h (x, x+h | \pi )+V^{h}_{n}  (x-h) p^h ( x, x-h | \pi \big)
\Big].
\eea
The iterations stop as soon as the
increment
$V^h_{n+1}\cd-V^h_n\cd$
%is below
reaches
some tolerance level. We set the error tolerance to be $10^{-8}$.

\begin{figure}[h!tb]
	\begin{center}
		\includegraphics[height=2.5in,width=4.5in]{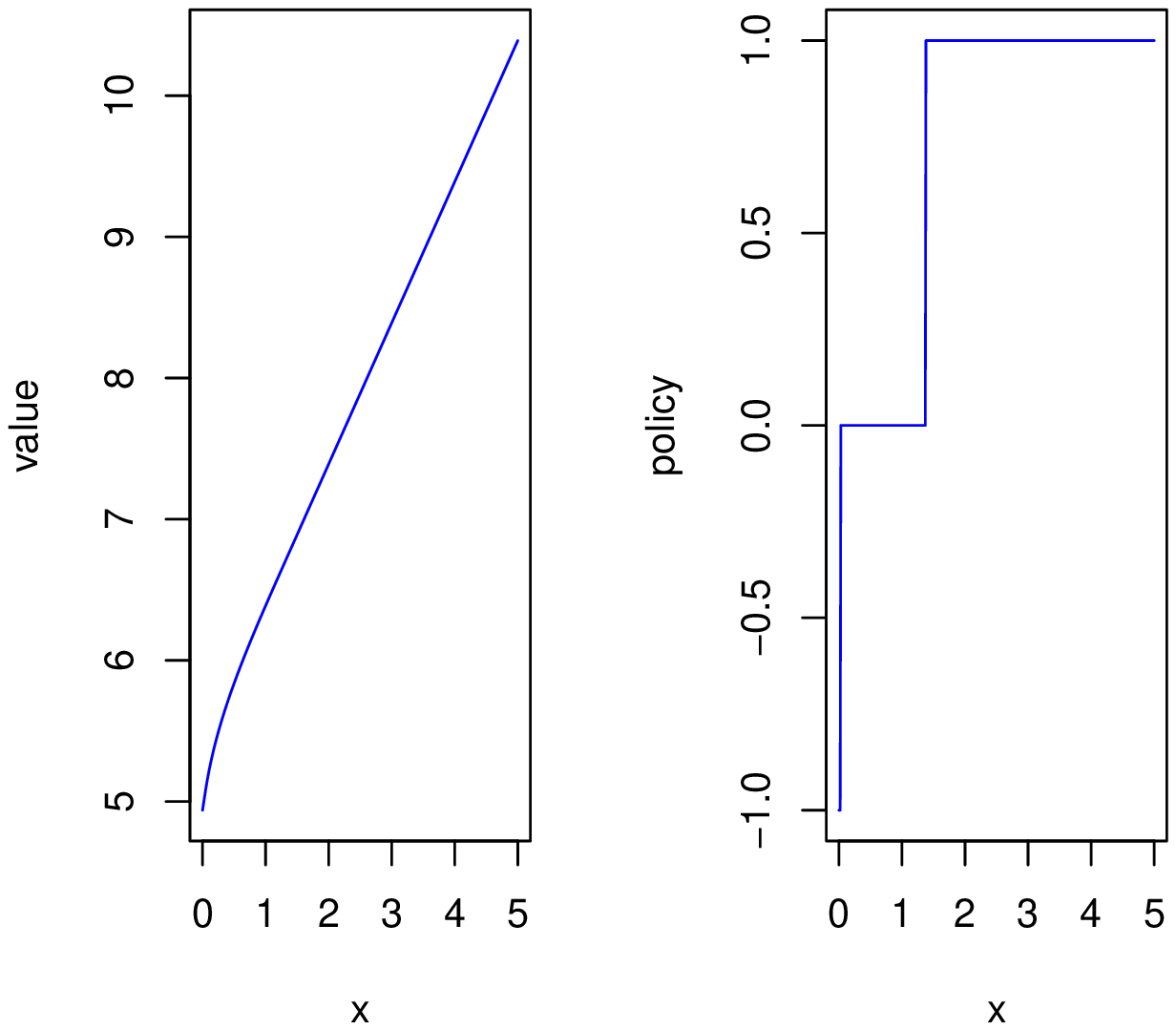}
		\caption{Value function and optimal policies: $f(x)=1, g(x)=3$ for $x\geq 0.$} \label{fig1}\end{center}
\end{figure}

\begin{figure}[h!tb]
	\begin{center}
		\includegraphics[height=2.5in,width=4.5in]{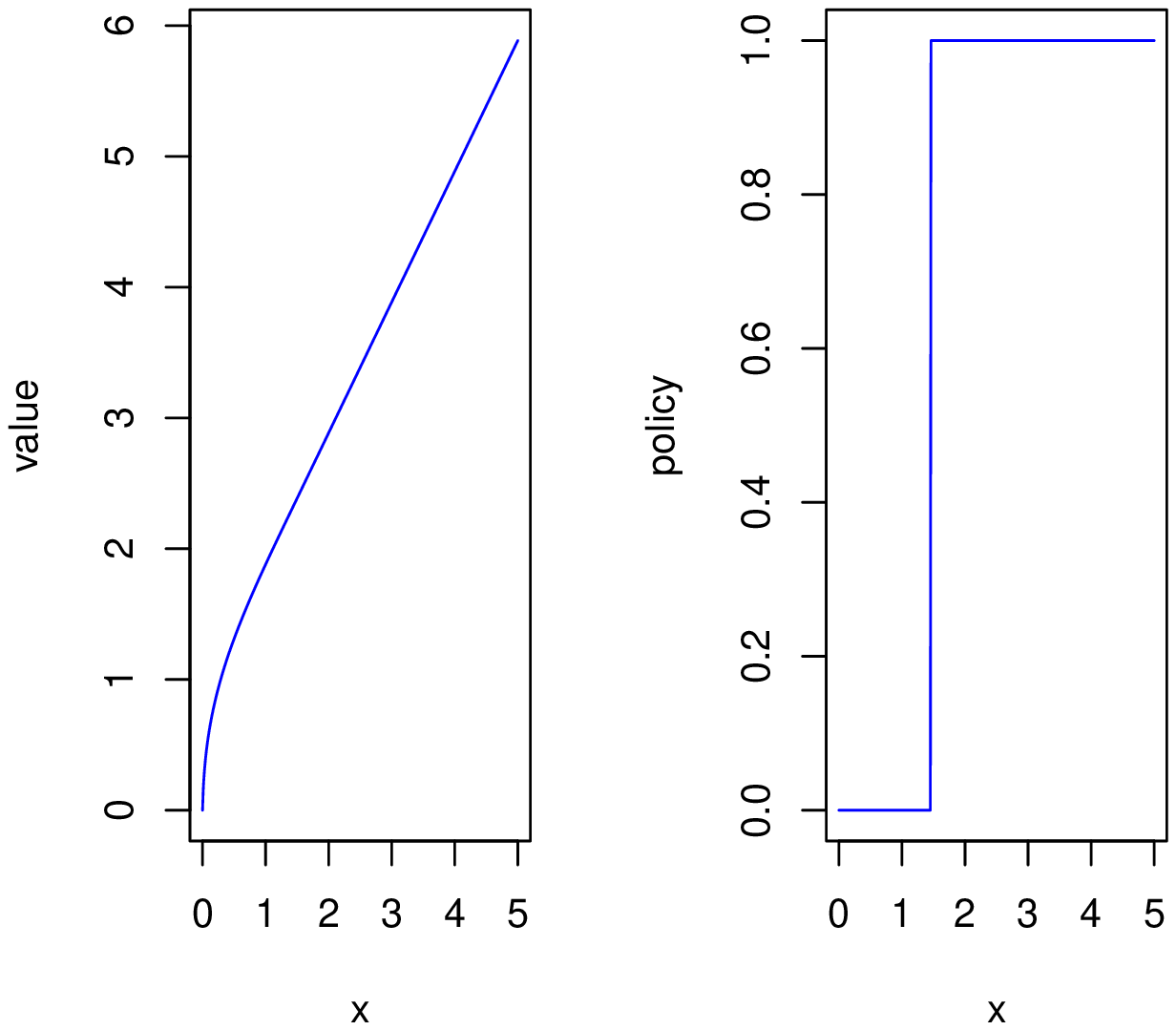}
		\caption{Value function and optimal policies: $f(x)=1, g(x)=50$ for $x\geq 0$.} \label{fig2}\end{center}
\end{figure}

\begin{figure}[h!tb]
	\begin{center}
		\includegraphics[height=2.5in,width=4.5in]{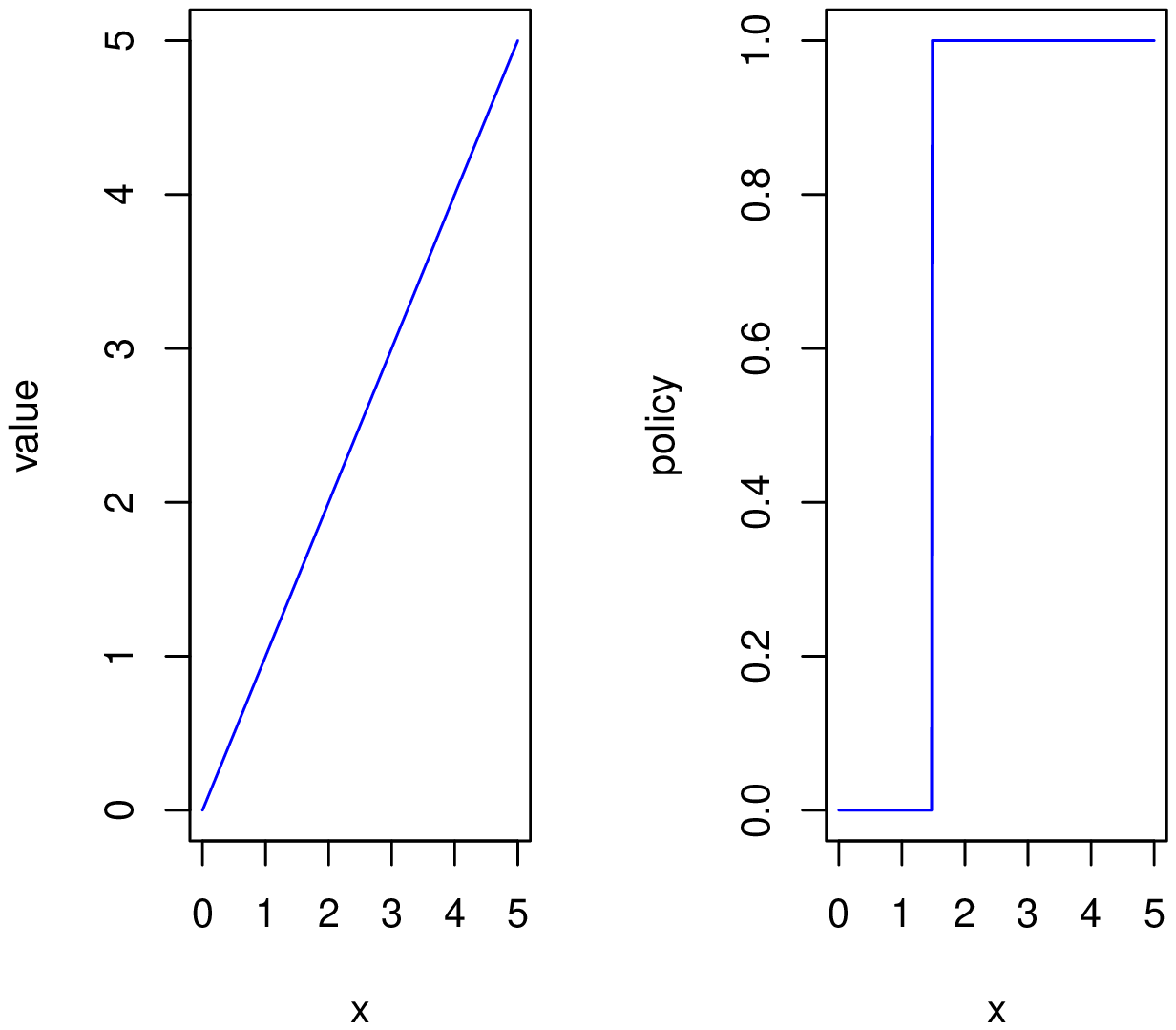}
		\caption{Value function and optimal policies: $f(x)=1, g(x)=3, \sigma(x)=1000$ for  $x\geq 0.$} \label{fig3}\end{center}
\end{figure}

\begin{figure}[h!tb]
	\begin{center}
		\includegraphics[height=2.5in,width=4.5in]{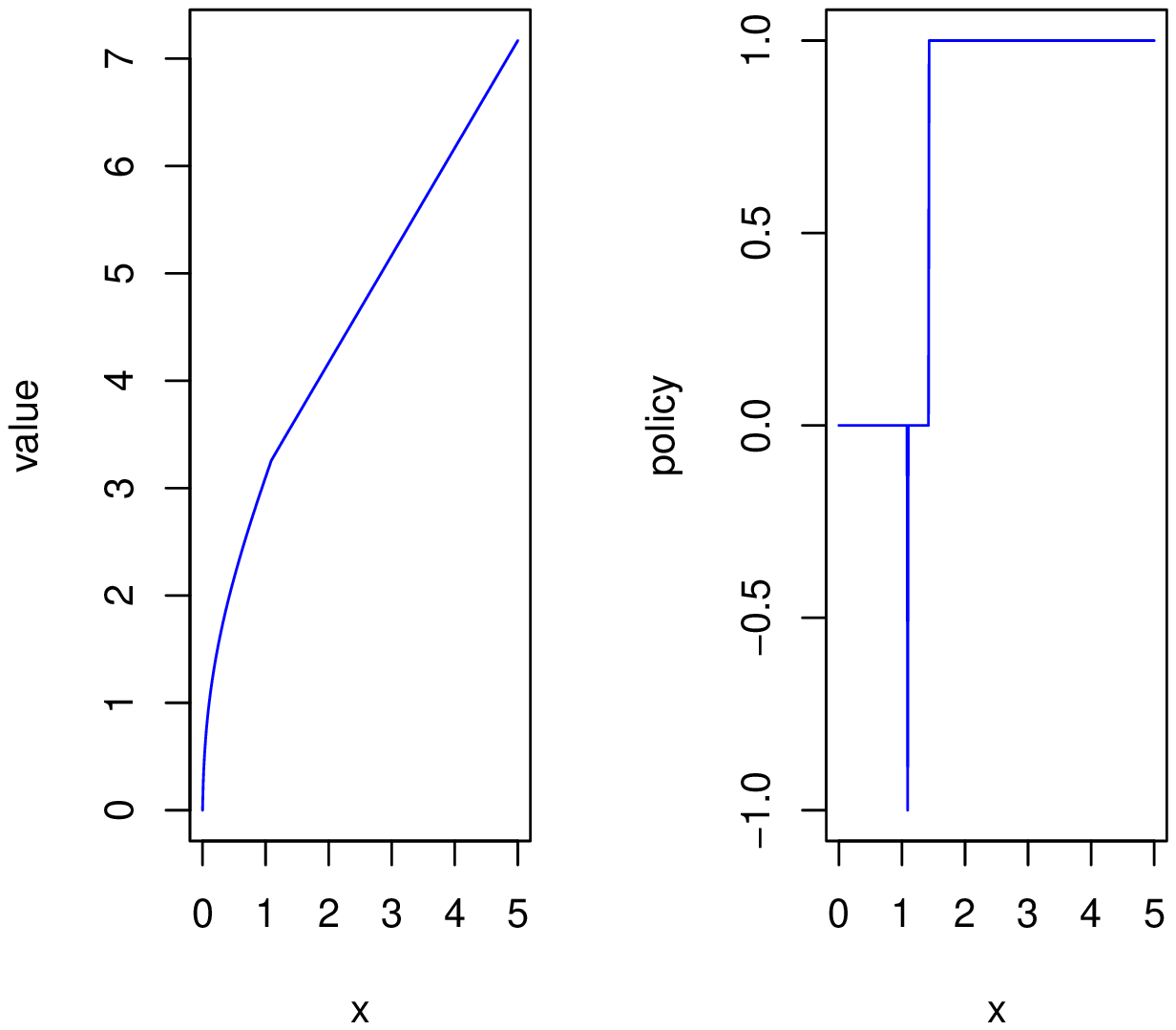}
		\caption{Value function and optimal policies: $f(x)=1$, $g(x)=100$ for $x\le 1$, $g(x)=1.02$ for $x>1.1$, and $g$ is an affine function on the interval $(1, 1.1).$} \label{fig4}\end{center}
\end{figure}

The numerical experiments provide evidence that the following conjecture holds

\begin{conj}
Suppose we have one species that evolves according to \eqref{e.6.1} and suppose Assumption \ref{a:1} holds. One can construct the optimal harvesting-seeding strategy $(Y^*,Z^*)$ as follows. There exist lower and upper thresholds $0\leq u^*< v^*<\infty$ such that after $t\geq 0$ the density of the species always stays in the interval $[u^*,v^*]$. More explicitly if $X(0-)=x$ then
\begin{equation}\label{e:optimal_strategy}
\begin{aligned}
(Y^{u^*}(t), Z^{v^*}(t))
&= \begin{cases}
\left((x-v^\ast)^+,(u^\ast-x)^+\right) & \mbox{if $t=0$,} \\
(L(t,v^\ast), L(t,u^\ast)) & \mbox{if $t>0$}
\end{cases}
\end{aligned}
\end{equation}
where $L(t,u^\ast)$ (respectively $L(t,v^\ast)$) is the local time push of the process $X$ at the boundary $u^*$ (respectively $v^*$).
\end{conj}

For the first numerical experiment we take $b=3,  c=2, \sg=1$ in \eqref{e.6.1}. Let $\delta=0.05$,  and $f(x)=1,g(x)=3$ for all $x\ge 0$.
Figure \ref{fig1} shows the value function $V(x)$ as a function of the initial population $x$ and
provides optimal policies, with $1$ denoting harvesting,  $-1$ denoting
seeding, and $0$ denoting no harvesting or seeding.
It can be seen from Figure \ref{fig1}  that  the optimal policy is a barrier strategy. There are levels $L_1$ and $L_2$ such that $[0, L_1)$ is the seeding region, $[L_1, L_2)$ is the diffusion region where there is no control of the population, and $[L_2, U]$ is the
harvesting region. Because of the benefit from seeding, $V(0)>0$.
These observations agree with those in the analogous financial setting \cite{Jin12, Scheer}.

Next, suppose that $g$ takes very large values. In particular, we take $g(x)=50, x\ge 0$. In this case, one can observe that there is no seeding; see Figure \ref{fig2}. In other words, because the cost of seeding is very high, the optimal strategy does not benefit from seeding.

To explore how noise impacts the problem, we explore what happens when $\sigma=1000$ and keep the other coefficients the same. The results are shown on Figure \ref{fig3}. It turns out, as expected, that if the noise is very large, the value function is close to the current harvesting potential $J(\cdot, Y_0, Z_0)$ and no seeding is needed. This is because the large noise will drive the species extinct with probability $1$ and, therefore, the optimal strategy is to immediately harvest all individuals. We refer to \cite{Alvarez98, Ky16, Hening, AH18} for more insight regarding how noise impacts harvesting.

We emphasize that the idea of species seedings is in part motivated by capital injections in optimal dividend problems. In \cite{Jin12,Scheer}, theoretical and experimental results show that  it
is optimal to have seeding (capital injections) only if the surplus hits zero or if it is smaller than a sufficiently low threshold.
In our formulation, both $f$ and $g$ can be density-dependent and this leads to new phenomena. To exhibit this, we take
$g(x)=50$ for $x\le 1$, $g(x)=1.1$ for $x>1.1$, and let $g$ be an affine function on $(1, 1.1)$. The results
from Figure \ref{fig4} tell us that we should only have seeding when the population has density $x=1.1$.

\subsection{Two-species ecosystems}

	\begin{figure}[h!tb]
		\begin{center}
			\includegraphics[height=4in,width=3.5in]{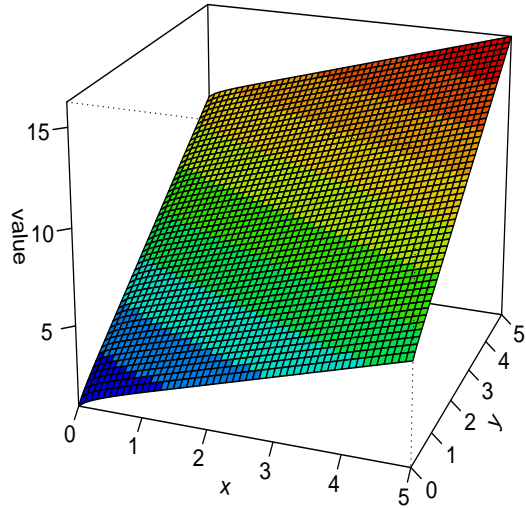}
			\caption{Value function vs initial population for a two-species competitive model.} \label{fig5}\end{center}
	\end{figure}
	
	\begin{figure}[h!tb]
		\begin{center}
			\includegraphics[height=4in,width=3.5in]{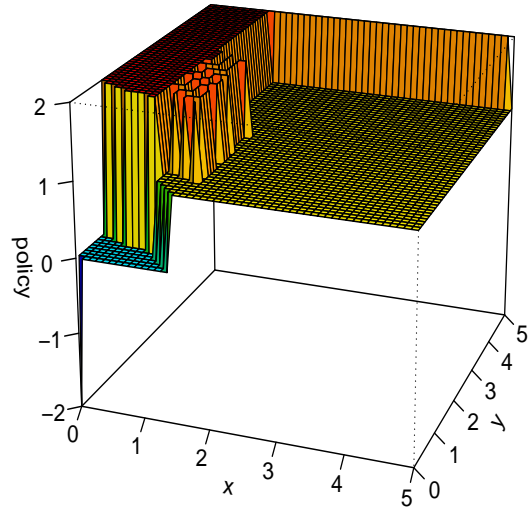}
			\caption{Optimal harvesting-seeding strategy vs population size for a two-species competitive model.} \label{fig6}\end{center}
	\end{figure}
	
	\begin{figure}[h!tb]
		\begin{center}
			\includegraphics[height=2.5in,width=4.5in]{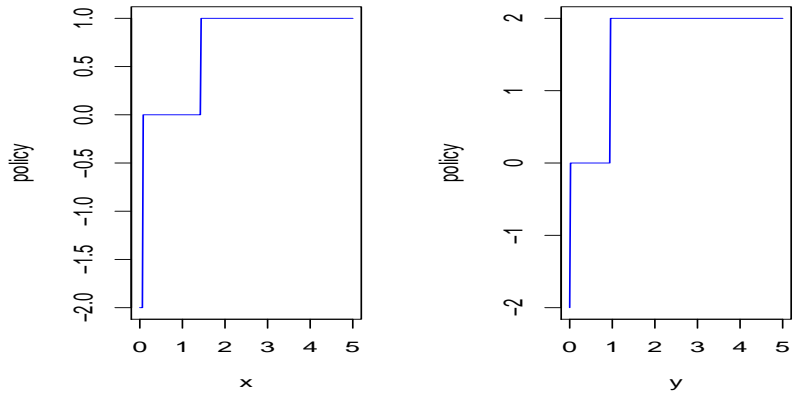}
			\caption{Optimal harvesting-seeding strategy vs population size for a two-species competitive model: the case $y=0$ (the left picture) and $x=0$ (the right picture).} \label{fig7}\end{center}
	\end{figure}
	
\begin{exm} {\rm 
		Consider two species competing according to the following stochastic Lotka-Volterra system 
		\beq{}\barray
		\aad d X_1(t)  = X_1(t)\Big(b_1-a_{11}X_1(t) - a_{12} X_2(t)   \Big) + \sg_1X_1(t)dw(t)-dY_1(t) + dZ_1(t)\\
		\aad d X_2(t)  = X_2(t)\Big(b_2-a_{21}X_1(t) - a_{22} X_2(t)   \Big) + \sg_2X_2(t)dw(t)-dY_2(t) + dZ_2(t)
		\earray
		\eeq
		and suppose that $\delta=0.05$, $f_1(x)=1$, $f_2(x)=2$, $g_1(x)=4, g_2(x)=4$ for all $x\in [0, \infty)^2$. We take
		$U = 5$. In addition, set
		\begin{equation*}
		b_1=3,  a_{11}=2, a_{12}=1,  \sg_1=3,  b_2=2,  a_{21}=1, a_{22}=2,  \sg_2=3.
		\end{equation*}
			Figure \ref{fig5} shows the value function $V$ as a function of initial population sizes $(x,y)$. 
		Figure \ref{fig6}  provides the optimal harvesting-seeding policies, with ``1" denoting harvesting of species 1,  ``-1" denoting
		seeding of species 1, ``2" denoting harvesting of species 2,  ``-2" denoting
		seeding of species 2 and ``0" denoting no harvesting or seeding.
		It can be seen from Figure \ref{fig6} that when population size of each species is larger than a certain level, it is optimal to harvest. However, for a very large region, harvesting of species 1 is the first choice. Moreover, one can observe that it is never optimal to seed species 1. As shown in
		Figure 7, if we assume both species densities are $0$ initially, we should only seed species 2. This tells us that the benefits obtained from species 2 are larger than those from species 1 due to its higher price; i.e, $f_2(x)=2>1=f_1(x)$. 
	}
\end{exm}
\begin{figure}[h!tb]
		\begin{center}
			\includegraphics[height=4in,width=3.5in]{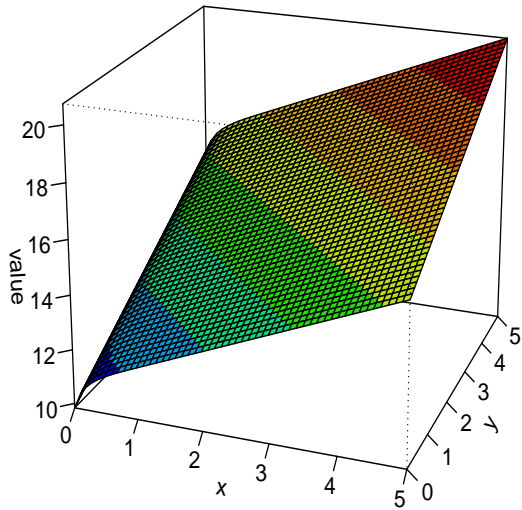}
			\caption{Value function vs initial population for a two-species predator-prey model.} \label{fig8}\end{center}
	\end{figure}
	
	\begin{figure}[h!tb]
		\begin{center}
			\includegraphics[height=4in,width=3.5in]{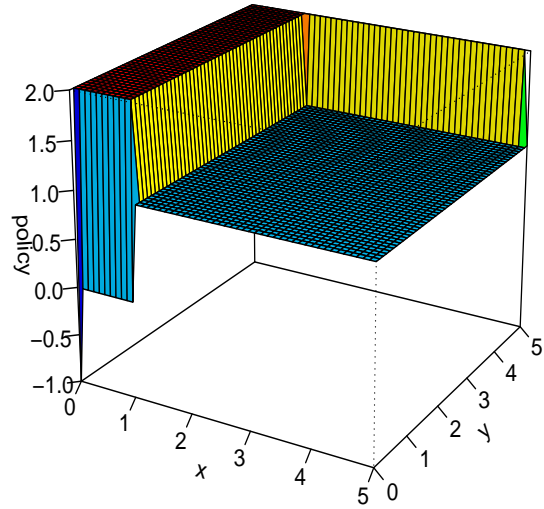}
			\caption{Optimal harvesting-seeding strategy vs population size for a two-species predator-prey model.} \label{fig9}\end{center}
	\end{figure}
	
	\begin{figure}[h!tb]
		\begin{center}
			\includegraphics[height=2.5in,width=4.5in]{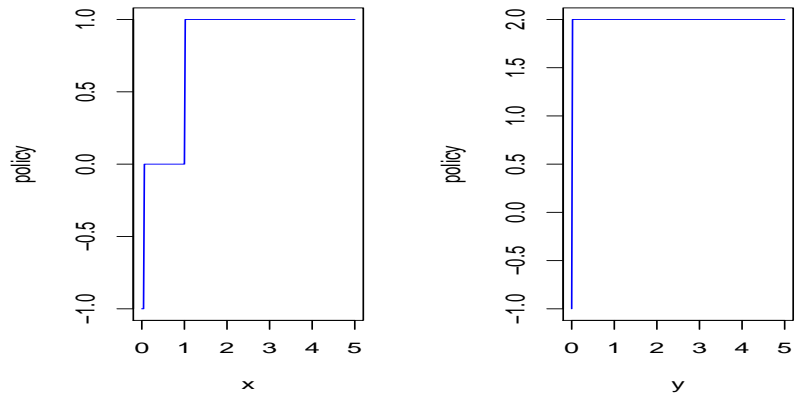}
			\caption{Optimal harvesting-seeding strategy vs population size for a two-species predator-prey model: the case $y=0$ (the left picture) and $x=0$ (the right picture).} \label{fig10}\end{center}
	\end{figure}
	
\begin{exm}{\rm 
		Consider a predator-prey system modelled by the stochastic Lotka--Volterra system
		\beq{}\barray
		\aad d X_1(t)  = X_1(t)\Big(b_1-a_{11}X_1(t) - a_{12} X_2(t)   \Big) + \sg_1X_1(t)dw(t)-dY_1(t) + dZ_1(t)\\
		\aad d X_2(t)  = X_2(t)\Big(-b_2 + a_{21}X_1(t) - a_{22} X_2(t)   \Big) + \sg_2X_2(t)dw(t)-dY_2(t) + dZ_2(t).
		\earray
		\eeq
		Conditions for the coexistence and extinction of the differenmt species can be found in \cite{HN18}.
		Suppose that $\delta=0.05$, $f_1(x)=1$, $f_2(x)=1$, $g_1(x)=6, g_2(x)=6$ for all $x\in [0, \infty)^2$. We take
		$U = 5$. In addition, let
		\begin{equation*}
		b_1=2,  a_{11}=1.2, a_{12}=1,  \sg_1=1.2,  b_2=1,  a_{21}=1.2, a_{22}=7,  \sg_2=1.3.
		\end{equation*}
		Figure \ref{fig8} shows the value function $V$ as a function of initial population $(x,y)$. 
		Figure \ref{fig9}  provides the optimal harvesting-seeding strategies, with ``1" denoting harvesting on species 1,  ``-1" denoting
		seeding on species 1, ``2" denoting harvesting on species 2,  ``-2" denoting
		seeding on species 2 and ``0" denoting no harvesting or seeding. We see in Figure \ref{fig10} that, as expected, since the predator goes extinct if there is no prey, the optimal strategy is to immediately harvest all the predators at time $t=0$.  
	}
\end{exm}

{\bf Acknowledgments.} KT was supported in part by Vietnam National Foundation for Science and Technology Development (NAFOSTED) under Grant 101.03-2017.23. GY was supported in part by the National Science Foundation under grant DMS-1710827.

\bibliographystyle{amsalpha}
\bibliography{harvest}

\appendix

\section{Properties of Value Functions}\label{sec:pro}

This section is devoted to several properties of the value function. Particularly, the lemma below will be helpful in proving the boundedness of the value function.

\begin{lem}\label{lem}
 Let Assumption \ref{a:1} be satisfied and pick $\Phi\cd\in C^2(\rr^d)$. Then for any $s\ge 0$, there exist  $\widehat{X}(s)\in \rr^d$ and $\widetilde{X}(s)\in \rr^d$ such that $\widehat{X}(s)\le X(s)$, $\widetilde{X}(s)\le X(s)$, and $$\Phi({X}(s) )- \Phi({X}(s-))= -\Delta Y(s)\cdot \nabla \Phi (\widehat{X}(s)) + \Delta Z(s)\cdot \nabla \Phi (\widetilde{X}(s)).$$
\end{lem}

\begin{proof}
	Without loss of generality, we suppose that $\Delta Y_i(s)> 0$ for $i=1, \dots, k$ and $\Delta Y_i(s)=0$ for $i=k+1, \dots, {d}$, where $k\le d$.
	Define
	$$ {X}^*(s)=\big(X_1(s-), \dots, X_k(s-), X_{k+1}(s), \dots, X_d(s)\big)'.$$
	
	Note that
	$\Delta Z_i(s)= 0$ for $i=1, \dots, k$ and $\Delta Y_i(s)=0$ for $i={k+1}, \dots, d$.
	We can check that $$X(s)-{X}^*(s)=-\Delta Y(s)\le 0 \quad \text{and} \quad {X}^*(s)-{X}(s-)=\Delta Z(s) \ge 0.$$
	By the mean value theorem, there is a point $\widehat{X}(s)\le X(s)$ on the line segment connecting ${X}(s)$ and ${X}^*(s)$ such that
	\beq{e.3.25}
	\Phi({X}(s) )- \Phi({X}^*(s)) = -\Delta Y(s)\cdot \nabla \Phi (\widehat{X}(s)).
	\eeq
	Similarly, there is a point $\widetilde{X}(s)\le X(s)$ on the line segment connecting ${X}(s-)$ and ${X}^*(s)$ such that
	\beq{e.3.26}
	\Phi({X}^*(s) )- \Phi({X}(s-)) = \Delta Z(s)\cdot \nabla \Phi (\widetilde{X}(s)).
	\eeq
	The conclusion follows from \eqref{e.3.25} and \eqref{e.3.26}.
\end{proof}

\begin{thm}\label{prop:a2}
	Let Assumption \ref{a:1} be satisfied. Suppose that there exists a function $\Phi : \Sb  \mapsto [0, \infty)$ such that
	$\Phi\in C^2(\Sb )$   and that $\Phi\cd$ solves the following coupled system of quasi-variational inequalities
	\beq{e.3.1}
	\sup\limits_{(x, i)} \bigg\{(\L-\delta)\Phi(x), f_i(x)-\dfrac{\partial \Phi}{\partial x_i}(x),  \dfrac{\partial \Phi}{\partial x_i}(x)-g_i(x)\bigg\}\le 0,
	\eeq
	where $(\L-\delta)\Phi(x)=\L \Phi(x)-\delta \Phi(x)$. Then the following assertions hold.
	\begin{itemize}
		\item[{\rm (a)}]  We have
		$$V(x)\le \Phi(x)\quad \text{for any } x\in {S}.$$
		\item[{\rm (b)}]   Define the non-intervention region 
		$$\mathcal{C}= \{x\in S:  f_i(x)<\dfrac{\partial \Phi}{\partial x_i}(x)<g_i(x) \}.$$
		Suppose that $(\L -r)\Phi(x)=0$ for all $x\in \mathcal{C}$, and that there exists a harvesting strategy $\big(\wdt{Y}, \wdt{Z}\big)
		\in \mathcal{A}_{x}$
		and a corresponding process $\wdt{X}$
		such that the following statements hold.
		
		\medskip
		
		\begin{itemize}
			\item[{\rm (i)}] $\wdt{X}(t)\in \mathcal{C}$ for Lebesgue almost all $t\ge 0.$
			\item[{\rm (ii)}] $\int\limits_0^{t} \Big[\nabla \Phi(\wdt{X}(s))- f(\wdt{X}(s))\Big] \cdot d\wdt{Y}^c(s)=0$ for any $t\ge
			0$.
			\item[{\rm (iii)}] $\int\limits_0^{t} \Big[g(\wdt{X}(s))-\nabla \Phi(\wdt{X}(s)) \Big] \cdot d\wdt{Z}^c(s)=0$ for any $t\ge
			0$.
			\item[{\rm (iv)}] $\lim\limits_{N\to\infty}E_{x}\Big[e^{-rT_N}\Phi(\wdt{X}(T_N))\Big]=0$, where for each $N=1, 2, \dots$,
			\beq{kyy}\beta_N:=\inf\{t\ge 0: |X(t)|\ge N\},\quad T_N: = N \wedge \beta_N.\eeq
			\item[{\rm (v)}] If $\wdt{X}(s)\ne \wdt{X}(s-)$, then
			$$\Phi(\wdt{X}(s))-\Phi(\wdt{X}(s-))=- f(\wdt{X}(s-)) \cdot \Delta \wdt{Z}(s)
			.$$
		\end{itemize}
		Then $V(x)=W(x)$ for all $x\in S$, and $\big(\wdt{Y},\wdt{Z}\big)$ is an optimal harvesting strategy.
	\end{itemize}
\end{thm}

\begin{proof}
	
	(a) Fix some $x \in S$ and $(Y, Z)\in \mathcal{A}_{x}$, and let $X$ denote the corresponding harvested process.  Choose $N$ sufficiently large so that $|x|<N$. For
	$$
	\beta_N=\inf\{t\ge 0:   |X(t)|\ge N\}, \quad T_N = N\wedge \beta_N,
	$$
	we have \beq{e.3.2}
	\beta_N\to \infty\quad \text{and}\quad T_N \to \infty \quad  \text{almost surely as } \quad N \to\infty.
	\eeq
	Then Dynkin's formula leads to
	\bea
	\E_{x} \ad \big[e^{-\delta T_N}\Phi\(X(T_N)\)\big]-\Phi(x)\\
	\ad = \E_{x}\int_{0}^{T_N} e^{-\delta s}(\L-\delta)\Phi\(X(s)\)ds-\E_{x}\int_{0}^{T_N} e^{-\delta s}\nabla \Phi\(X(s)\)\cdot dY^c(s)\\
	\ad\quad +\E_{x}\int_{0}^{T_N} e^{-\delta s}\nabla \Phi\(X(s)\)\cdot dZ^c(s)\\ \ad \quad + \E_{x}\sum\limits_{0\le s\le T_N}e^{-\delta s}\Big[\Phi\(X(s)\)-\Phi\(X(s-)\)\Big].
	\eea
	It follows from \eqref{e.3.1} that
	\beq{e.3.3}\barray
	\E_{x} \ad \big[e^{-\delta T_N}\Phi\(X(T_N)\)\big]-\Phi(x)\\
	\ad \le -\E_{x}\int_{0}^{T_N} e^{-\delta s}\nabla \Phi\(X(s)\)\cdot dY^c(s) + \E_{x}\int_{0}^{T_N} e^{-\delta s}\nabla \Phi\(X(s)\)\cdot dZ^c(s)\\
	\ad\quad  + \E_{x}\sum\limits_{0\le s\le T_N}e^{-\delta s}\Delta \Phi\(X(s)\),
	\earray
	\eeq
	where $\Delta \Phi\(X(s)\)=\Phi\(X(s)\)-\Phi\(X(s-)\)$.
	By virtue of Lemma \ref{lem}, the monotonicity of
	$f(\cdot)$, $g(\cdot)$, and \eqref{e.3.1},
	we obtain
	\beq{e.3.4}\Delta \Phi(X(s))\le
	-f\(X(s-)\)\cdot \Delta Y(s)) +g(X(s-))\cdot \Delta Z(s).\eeq
	Since $\Phi\cd$ is nonnegative, it follows from \eqref{e.3.3} and \eqref{e.3.4} that
	\bea
	\Phi(x) \ad \ge \bigg[ \E_{x}\int_{0}^{T_N} e^{-\delta s}f(X(s-))\cdot dY^c(s)
	+ \E_{x}\sum\limits_{0\le s\le T_N}e^{-\delta s}f(X(s-))\cdot \Delta Y(s)\bigg] \\
	\ad \quad  - \bigg[ \E_{x}\int_{0}^{T_N} e^{-\delta s}g(X(s-))\cdot dZ^c(s)
	+ \E_{x}\sum\limits_{0\le s\le T_N}e^{-\delta s}g(X(s-))\cdot \Delta Z(s)\bigg] \\
	\ad = \E_{x}\int_{0}^{T_N} e^{- \delta s}f(X(s-))\cdot dY(s)-\E_{x}\int_{0}^{T_N} e^{- \delta s}g(X(s-))\cdot dZ(s).
	\eea
	Letting $N\to\infty$, it follows from \eqref{e.3.2} and the bounded convergence theorem that
	$$\Phi(x)\ge \E_{x}\bigg[\int_0^\infty e^{-\delta s}f(X(s-))\cdot dY(s)-\int_0^\infty e^{-\delta s}g(X(s-))\cdot dZ(s)\bigg].$$
	Taking supremum over all $(Y, Z)\in \mathcal{A}_{x}$, we obtain $\Phi(x)\ge V(x)$.
	
	(b) 	Let (i)-(v) be satisfied. 
	Then Dynkin's formula leads to
	\bea
	\E_{x} \ad \big[e^{-\delta T_N}\Phi(\wdt{X}(T_N))\big]-\Phi(x)\\
	\ad = \E_{x}\int_{0}^{T_N} e^{-\delta s}(\L-\delta)\Phi(\wdt{X}(s))ds-\E_{x}\int_{0}^{T_N} e^{-\delta s}\nabla \Phi(\wdt{X}(s))\cdot d\wdt{Y}^c(s)\\
	\ad\quad +\E_{x}\int_{0}^{T_N} e^{-\delta s}\nabla \Phi(\wdt{X}(s))\cdot d\wdt{Z}^c(s) + \E_{x}\sum\limits_{0\le s\le T_N}e^{-\delta s}\Big[\Phi(\wdt{X}(s))-\Phi(\wdt{X}(s-))\Big].
	\eea
	By (i), $(\L-\delta)\Phi(\wdt{X}(s))$ for almost all $s\ge 0$. This, together with (ii) and (iv), implies that
	\bea
	\Phi(x)\ad=\E_{x} \big[e^{-\delta T_N}\Phi\big(\wdt{X}(T_N)\big)\big]+\E_{x}\int_{0}^{T_N} e^{- \delta s}f(\wdt{X}(s-))\cdot d\wdt{Y}(s)\\
	\ad \qquad -\E_{x}\int_{0}^{T_N} e^{- \delta s}g(\wdt{X}(s-))\cdot d\wdt{Z}(s).\eea
	Letting $N\to \infty$ and using (iv), we obtain
	$$
	\Phi(x)=\E_{x}\int_{0}^{T_N} e^{- \delta s}f(\wdt{X}(s-))\cdot d\wdt{Y}(s) -\E_{x}\int_{0}^{T_N} e^{- \delta s}g(\wdt{X}(s-))\cdot d\wdt{Z}(s).$$
	This, together with (a), implies that $\Phi(x)=V(x)$ for any $x\in S$ and $(\wdt{Y}, \wdt{Z})$ is an optimal harvesting strategy.	
\end{proof}

Next, we establish the continuity of the value function.

\begin{prop} \label{prop:2} Let Assumption \ref{a:1} be satisfied. Then the following assertions hold.
	\begin{itemize}
	\item[\rm (a)] For any $x, y\in \Sb $,
	\beq{e.3.5}V(y)\le V(x) - f(x)\cdot (x-y)^++g(x)\cdot (x-y)^-.\eeq
	
	\item[\rm (b)] If $V(0)<\infty$, then $V(x)<\infty$ for any $x\in \Sb $ and $V$ is Lipschitz continuous.
	\end{itemize}
\end{prop}

\begin{proof}\
	
(a) Fix  $(Y, Z)\in \mathcal{A}_{y}$. Define
$$\widetilde Y(t)=Y(t)+(x-y)^+, \quad \widetilde Z(t)=Z(t) + (x-y)^-, \quad t\ge 0.$$
Then $(\widetilde{Y},  \widetilde{Z})\in \mathcal{A}_{x}$. Let $\widehat{X}$ denote the process satisfying \eqref{e.2} with  $\widehat{X}(0-)=x$ and strategy $(Y, Z)$. Let $\widetilde{X}$ denote the process satisfying \eqref{e.2} with $\widetilde{X}(0-)=y$ and strategy $(\widetilde{Y}, \widetilde{Z})$. Then we have $\widehat{X}(t)=\widetilde{X}(t)$ for any $t> 0$. Consequently, it follows that
$$J(x, \widetilde{Y}, \widetilde{Z})=f(x)\cdot (x-y)^+-g(x)\cdot (x-y)^- + J(y, Y, Z).$$
 	Since $V(x)\ge J(x, \widetilde{Y}, \widetilde{Z})$, we have
 	$$V(x)\ge f(x)\cdot (x-y)^+-g(x)\cdot (x-y)^- + J(y, Y, Z),$$
	from which, \eqref{e.3.5}  follows by taking supremum over $(Y, Z)\in \mathcal{A}_{y}$.
	
	(b) Similar to \eqref{e.3.5}, we have
	\beq{e.3.6}V(x)\le V(y) - f(y)\cdot (y-x)^++g(y)\cdot (y-x)^-.\eeq
	In view of \eqref{e.3.5} and \eqref{e.3.6}, if  $V(0)<\infty$, then $V(x)<\infty$ for any $x\in \Sb $. Moreover,
	\bea|V(x)-V(y)|\ad \le |f(x) + f(y) + g(x)+g(y)| |x-y|\\
	\ad \le 2|f(0)+g(0)| |x-y|, \quad x, y\in S.\eea
	Thus, $V(\cdot)$ is Lipschitz continuous.
\end{proof}

Using Proposition \ref{prop:a2}, we proceed to present an easily verifiable
condition for the finiteness of the value function.

\begin{prop}
	\label{prop:3} Let Assumption \ref{a:1} be satisfied.  Moreover, suppose that there is a positive constant $C$ such that \begin{equation}\label{cond-bi}b_i(x)\le \delta x_i+C, \quad x\in\Sb , \quad i=1, \dots, d.\end{equation}
	Then there exists a positive constant $M$ such that
	$$V(x)\le \sum\limits_{i=1}^d f_i(0) x_i+ M, \quad x\in\Sb .$$
\end{prop}

\begin{proof}
	Define
	$$\Phi(x)=\sum\limits_{i=1}^d f_i(0) x_i + M, \quad x\in\Sb ,$$
	where $M$ is a positive number to be specified.
We can check that $\Phi\cd$ solves the system of inequalities \eqref{e.3.1}  for sufficiently large $M.$
By virtue of Proposition \ref{prop:a2}, $V(x)\le \Phi(x)$ for any $x \in S.$ The details are omitted for brevity.
\end{proof}

Throughout the rest of this section,
we aim to characterize the value function as a viscosity solution of an associated system of quasi-variational inequalities.
Our approach is motivated by \cite{Taksar, Zhu11}. However, the results and proofs below are nontrivial extensions because we have interacting species as well as seeding.
We use the following notation and definitions. For a point $x^0\in S$ and a strategy $(Y, Z)\in \mathcal{A}_{x^0}$,
let ${X}$ be the corresponding process with harvesting and seeding. Let $B_\e(x^0)=\{x\in S: |x-x^0|<\e \}$, where $\e>0$ is sufficiently small so that
$\overline{B_\e(x_0)}\subset S$.
Let $\theta=\inf\{t\ge 0: {X}(t)\notin B_\e(x^0) \}$. For a constant $r>0$, we  define $\theta_r=\theta\wedge r$.

\begin{prop}
Let Assumption \ref{a:1} be satisfied and suppose that $V(x)<\infty$ for all $x\in \Sb $.
The value function
$V$ is a viscosity subsolution of the system of quasi-variational inequalities
\beq{e.3.17}
\max\limits_{i} \bigg\{(\L-\delta)\Phi(x), f_i(x)-\dfrac{\partial \Phi}{\partial x_i}(x),  \dfrac{\partial \Phi}{\partial x_i}(x)-g_i(x)\bigg\}=0, \quad x\in S.
\eeq
That is,
	for
	any $x^0\in S$ and any function $\phi\in C^2(S)$ satisfying
	$$(V-\phi)(x)\ge (V-\phi)(x^0)=0,$$
	for all $x$ in a neighborhood of $x^0$, we have
	\beq{e.3.18}\max\limits_{i} \bigg\{(\L-\delta)\phi(x^0), f_i(x^0)-\dfrac{\partial \phi}{\partial x_i}(x^0),  \dfrac{\partial \phi}{\partial x_i}(x^0)-g_i(x^0)\bigg\}\le 0.\eeq
\end{prop}

\begin{proof}
For $x^0\in S$, consider
a $C^2$ function $\phi(\cdot)$ satisfying $\phi(x^0)=V(x^0)$ and
$\phi(x)\le V(x)$ for all $x$ in a neighborhood of $x^0$.
	Let  $\e>0$ be sufficiently small so that
 $\overline{B_\e(x_0)}\subset S$ and
 $\phi(x)\le V(x)$ for all $x\in \overline{B_\e(x_0)}$, where $\overline{B_\e(x_0)}=\{x\in S: |x-x^0|\le \e \}$ is the closure of $B_\e(x^0)$.

 Choose $(Y, Z)\in \mathcal{A}_{x^0}$ such that $Y(0-)=Z(0-)=0$, $Y(t)=Y(0)$ and $Z(t)=Z(0)$ for any $t\ge 0$, $|\Delta Y(0)|+|\Delta Z(0)|\le \eta$, where $\eta\in [0, \e)$. Thus, there are only jumps at time $t=0$.
	Let ${X}$ be the corresponding harvested process with initial condition $x^0$ and strategy $(Y, Z)$.

	Note that the chosen strategy $(Y, Z)$ guarantees that ${X}$ has at most one jump at $t=0$ and remains continuous on $(0, \infty)$. This, together with the fact that $\eta\in [0, \e)$, implies that
	 ${X}(t)\in \overline{B_\e(x^0)}$ for all $0\le t\le  \theta$. By virtue of the dynamic programming principle, we have
	\beq{e.3.19}
	\barray
	\phi(x^0)\ad = V(x^0)\\
	\ad \ge \E\bigg[ \int_0^{\theta_r} e^{-\delta s}f\({X}(s-)\)\cdot dY(s) - \int_0^{\theta_r} e^{-\delta s}g\({X}(s-)\)\cdot dZ(s) + e^{-\delta \theta_r}   V({X}(\theta_r))\bigg]\\
	\ad \ge \E\bigg[ \int_0^{\theta_r} e^{-\delta s}f\({X}(s-)\)\cdot dY(s) - \int_0^{\theta_r} e^{-\delta s}g\({X}(s-)\)\cdot dZ(s)  + e^{-\delta\theta_r}   \phi({X}\theta_r))\bigg].
	\earray
	\eeq
	By the Dynkin formula, we obtain
	\beq{e.3.20}
	\barray
	\phi(x^0)\ad = \E e^{-\delta \theta_r} \phi ({X}(\theta_r)) - \E \int_0^{\theta_r} e^{-\delta s} (\L -\delta ) \phi ({X}(s))ds\\
	\ad \qquad + \E \int_0^{\theta_r} e^{-\delta s}  \nabla \phi ({X}(s)) \cdot dY^c (s) -  \E\int_0^{\theta_r} e^{-\delta s}  \nabla \phi ({X}(s)) \cdot dZ^c (s) \\
	\ad \qquad - \E \sum\limits_{0\le s\le \theta_r} e^{-\delta s} \Big[ \phi({X}(s) )- \phi({X}(s-)) \Big].
	\earray
	\eeq
	A combination of \eqref{e.3.19} and \eqref{e.3.20} leads to
	\beq{e.3.21}
	\barray
	0 \ad \ge \E  \int_0^{\theta_r} e^{-\delta s}f\({X}(s-)\)\cdot dY(s) - \E  \int_0^{\theta_r} e^{-\delta s}g\({X}(s-)\)\cdot dZ(s)\\

	\ad 	+ \E \int_0^{\theta_r} e^{-\delta s} (\L -\delta ) \phi ({X}(s))ds\\
	\ad \qquad - \E \int_0^{\theta_r} e^{-\delta s}  \nabla ({X}(s)) \cdot dY^c (s) + \E\int_0^{\theta_r} e^{-\delta s}  \nabla ({X}(s)) \cdot dZ^c (s) \\
	\ad \qquad + \E \sum\limits_{0\le s\le \theta_r} e^{-\delta s} \Big[ \phi\({X}(s)\)- \phi\({X}(s-)\) \Big].
	\earray
	\eeq
First, we take $\eta=0$; that is, $Y(t)=Z(t)=0$ for any $t\ge 0$.
	Then
	$\theta>0$
almost surely
 (a.s.) and
	\eqref{e.3.21} can be rewritten as
	\beq{e.3.21a}
	\barray
	0\ad \ge \E \int_0^{\theta_r} e^{-\delta s} (\L -\delta ) \phi ({X}(s))ds.
	\earray
	\eeq
	We suppose that $(\L-\delta)\phi (x^0)>0$. Then we can choose a sufficiently small $\e>0$ such that $(\L-\delta)\phi (x)>0$ for any $x\in B_\e(x^0)$. As a result, $(\L-\delta)\phi (X(s))>0$ for any $t\in [0, \theta)$. It follows that $\int_0^{\theta_r} e^{-\delta s} (\L -\delta ) \phi ({X}(s))ds>0$ and therefore,
	$$\E \int_0^{\theta_r} e^{-\delta s} (\L -\delta ) \phi ({X}(s))ds>0,$$
	which contradicts \eqref{e.3.21a}. This proves that
	\beq{e.3.22}
	(\L-\delta)\phi(x^0)\le 0.
	\eeq
	Next, we fix $i\in \{1, \dots, d\}$ and take $\eta\in (0, \e)$, $Y_i(t)=Y_i(0)=\eta$ for all $t\ge 0$, and $Y_j(t)=Z_i(t)=Z_j(t)=0$ for all $j\ne i$ and $t\ge 0$.
	Then \eqref{e.3.21} reduces to
	$$\E \int_0^{\theta_r} e^{-\delta s} (\L -\delta ) \phi ({X}(s))ds + f_i(x^0)\eta + \phi(x^0- \eta {\bf e}_i)-\phi(x^0)\le 0.$$
	Now sending $r\to 0$, we have
	$$f_i(x^0)\eta + \phi(x^0- \eta {\bf e}_i)-\phi(x^0)\le 0.$$
	By dividing the above inequality by $\eta$ and letting $\eta\to 0$, we obtain
	\beq{e.3.23}
	f_i(x^0)-\partial \phi/\partial x_i (x^0)\le 0.
	\eeq
	Now we take $\eta\in (0, \e)$, $Z_i(t)=\eta$ for all $t\ge 0$, and $Z_j(t)=Y_i(t)=Y_j(t)=0$ for all $j\ne i$ and $t\ge 0$.
	Then \eqref{e.3.21} reduces to
	$$\E \int_0^{\theta_r} e^{-\delta s} (\L -\delta ) \phi ({X}(s))ds - g_i(x^0)\eta + \phi(x^0+ \eta {\bf e}_i)-\phi(x^0)\le 0.$$
	Now sending $r\to 0$, we have
	$$-g_i(x^0)\eta + \phi(x^0- \eta {\bf e}_i)-\phi(x^0)\le 0.$$
	Finally, dividing the above inequality by $\eta$ and letting $\eta\to 0$, we arrive at
	\beq{e.3.24}
	\partial \ph/\partial x_i(x^0)-g_i(x^0)\le 0.
	\eeq
	Now \eqref{e.3.18} follows by combining \eqref{e.3.22}, \eqref{e.3.23}, and \eqref{e.3.24}.
\end{proof}

\begin{lem}\label{lem:1} Let $\lambda$ be the random variable
defined as follows.
If $X(\theta)=X(\theta-)$ then $\lambda=0$, while if $X(\theta)\notin \overline{B_\e(x^0)}$, then let $\lambda$
be a positive number in $(0,1]$
	such that $${X}(\theta-) +\lambda ({X}(\theta) - {X}(\theta-)  )\in \partial {B_\e(x^0)}.$$ Then
		there is a positive number
		 $\kappa_0>0$ such that
		\beq{e.3.27c}
		\barray
\ad \E\Bigg[ \int_0^{\theta} e^{-\delta s} ds  + \lambda e^{-\delta \theta} \one \cdot\Delta Z(\theta)  + \lambda e^{-\delta \theta} \one \cdot\Delta Y(\theta) \\
\ad \qquad \qquad \qquad    + \int_0^{\theta -} e^{-\delta s}\one \cdot  dY(s) +  \int_0^{\theta -} e^{-\delta s}\one\cdot  dZ(s)\Bigg]\ge \ka_0,
\earray
\eeq
where $\one(x)=1$ for all $x\in\bar S$.
\end{lem}

\begin{proof} Recall that $\theta_r=\theta\wedge r$ for any positive number $r$. Define
	$${X}_\lambda^r:={X}(\theta_r-) +\lambda ({X}(\theta_r) - {X}(\theta_r-)  )={X}(\theta_r-)-\lambda (\Delta Y(\theta_r )- \Delta Z(\theta_r)).$$ It can be seen that $X_\lambda^r\in \overline{B_\e(x^0)}$ for any $r>0$.
	 We consider the function $\widetilde{\Phi} (x)= |x-x^0|^2-\varepsilon^2$ for $x \in B_\e(x^0)$. It follows that
	 $$(\L-\delta)\widetilde{\Phi}(x)=2(x-x^0)\cdot b(x) +\sum\limits_{i=1}^d \sigma_{ii}^2(x)-\delta(|x-x^0|^2-\varepsilon^2).$$
	 Since $\widetilde{\Phi}\cd, b\cd$, and $\sigma\cd$ are continuous, it is obvious that
	 $$|(\L-\delta )\widetilde{\Phi}(x)|\le K<\infty$$
	 for some positive constant $K$.
	 Note that $K$ depends only on $x^0, \e, \delta$ and bounds on $b\cd, \sg\cd$.
	  Let $K_0=\dfrac{1}{K+2\e}$ and define $\Phi(x)=K_0 \widetilde{\Phi}(x)$ for $x\in B_\e(x^0)$. Then it follows immediately that
	 \beq{e.3.27d}
	 |(\L-\delta)\Phi(x)|<1 \quad \text{for } x\in B_\e(x^0).
	 \eeq
	 Moreover, we have
	 \beq{e.3.27e}
	 \nabla \Phi(x) \cdot \one=2K_0 \cdot (x-x_0)\ge -1.
	 \eeq
	 By virtue of the Dynkin formula, we have
	 \beq{e.3.27f}
	 \barray \ad \E e^{-\delta \theta_r} \Phi({X}(\theta_r-)-\Phi(x^0) \\
	 \ad \ = \E \int_0^{\theta_r -} e^{-\delta s} (\L -\delta ) \Phi ({X}(s))ds - \E \int_0^{\theta_r-} e^{-\delta s}  \nabla \Phi ({X}(s)) \cdot dY^c (s)\\
	 \ad \qquad + \E\int_0^{\theta_r-} e^{-\delta s}  \nabla \Phi ({X}(s)) \cdot dZ^c (s) +  \E \sum\limits_{0\le s< \theta_r} e^{-\delta s} \Big[ \Phi({X}(s) )- \Phi({X}(s-)) \Big]
	. \earray
	 \eeq
	 By virtue of \eqref{e.3.27e}, we have
	 \beq{e.3.27g}
	 \barray
	 \ad 	\Phi({X}(s) )- \Phi({X}(s-))\\
	 \aad\ = ({X}(s) - ({X}(s-) )\cdot \nabla \Phi (P(s))
	 \\
	 \aad\ = (-\Delta Y (s) + \Delta Z(s)  ) \cdot \nabla \Phi (P(s)),\\
	 \aad\ \le \one\cdot (\Delta Y (s) + \Delta Z(s)  )
	 ,\earray
	 \eeq
	 where $P(s)$ is a point on the line connecting $X(s)$ and $X(s-)$.
	 Hence it follows from \eqref{e.3.27d}-\eqref{e.3.27g} that
	 \beq{e.3.27h}
	 \barray \ad \E e^{-\delta \theta_r} \Phi({X}(\theta_r-)))-\Phi(x^0)) \\
	 \ad \quad \le  \E \int_0^{\theta_r} e^{-\delta s}ds + E \int_0^{\theta_r-} e^{-\delta s}  \one\cdot dY^c (s)+\E \int_0^{\theta_r-} e^{-\delta s}  \one\cdot dZ^c (s)\\
	 \ad \qquad  +  \E \sum\limits_{0\le s< \theta_r} e^{-\delta s} \one\cdot\big[ \Delta Y (s) + \Delta Z(s)  \big]\\
	 \ad\quad = \E \int_0^{\theta_r} e^{-\delta s}ds + \E \int_0^{\theta_r-} e^{-\delta s}  \one\cdot dY (s) + \E \int_0^{\theta_r-} e^{-\delta s}  \one \cdot dZ(s)
	 .\earray
	 \eeq
	 We also have
	 \beq{}
	 \barray
	 \ad 	\Phi({X}(\theta_r-) )- \Phi({X}_\lambda^r
	 )\\
	 \aad\ = ({X}(\theta_r-) - {X}_\lambda^r )\cdot \nabla \Phi (P_0)
	 \\
	 \aad\ = \lambda(\Delta Y (\theta_r) - \Delta Z(\theta_r)  ) \cdot \nabla \Phi (P_0)\\
	 \aad\ \ge -\lambda \one\cdot (\Delta Y (\theta_r) + \Delta Z(\theta_r)  ),
	 \earray
	 \eeq
	 where $P_0$ is a point on the line segment connecting  ${X}(\theta_r-)$ and ${X}_\lambda^r$. Thus,
	 \beq{e.3.27i}
	 \E e^{-\delta \theta_r} \Phi({X}_\lambda^r
	 ) - \E e^{-\delta \theta_r}	\Phi({X}(\theta_r-) )\le \lambda \E e^{-\delta \theta_r} \one\cdot (\Delta Y (\theta_r) + \Delta Z(\theta_r)  ).
	 \eeq	
	 Combining \eqref{e.3.27h} and \eqref{e.3.27i}, we have
	 \beq{}\barray \ad \E e^{-\delta \theta_r} \Phi({X}_\lambda^r
	 )-\Phi(x^0)\\
\aad \ \le \E \int_0^{\theta_r} e^{-\delta s}ds + \E \int_0^{\theta_r-} e^{-\delta s}  \one\cdot dY (s) + \E \int_0^{\theta_r-} e^{-\delta s}  \one\cdot dZ(s) \\
	 \aad \qquad + \lambda \E e^{-\delta \theta_r} \one\cdot (\Delta Y (\theta_r) + \Delta Z(\theta_r)  ).
	 \earray
	 \eeq
By letting $r\to \infty$,
%[??Notation conflict, $h$ used in the definition of operator, here %$h\to \infty$, and later, $h$ as stepsize?? Ky's reply: I have %changed $h$ in the operator to $\Phi$, $h$ in this part to $r$. %Thanks.]
we arrive at
\beq{}\barray \ad \E \int_0^{\theta} e^{-\delta s}ds + \E \int_0^{\theta-} e^{-\delta s}  \one\cdot dY (s) + \E \int_0^{\theta-} e^{-\delta s}  \one\cdot dZ(s) \\
\ad \qquad + \lambda \E e^{-\delta \theta} \one\cdot (\Delta Y (\theta) + \Delta Z(\theta)  )\ge \E e^{-\delta \theta} \Phi({X}_\lambda
)-\Phi(x^0).
\earray
\eeq
If $\P(\theta=\infty)>0$, then $\E e^{-\delta \theta}\Phi({X}_\lambda)=0$. Otherwise, $\theta<\infty$ a.s. and in that case, since $\Phi({X}_\lambda)\in \partial B_\e(x^0)$,
 $\Phi({X}_\lambda)=0.$ Also, it is clear that $\Phi(x^0)= -K_0 \e^2.$ Thus,
we obtain
	 \bea\ad \E \int_0^{\theta} e^{-\delta s}ds + \E \int_0^{\theta-} e^{-\delta s}  \one\cdot dY (s) + \E \int_0^{\theta-} e^{-\delta s}  \one\cdot dZ(s)\\
	 \ad \qquad  + \lambda \E e^{-\delta \theta} \one\cdot (\Delta Y (\theta) + \Delta Z(\theta)  ) \ge K_0\e^2=\ka_0>0.
	 \eea
	 This establishes \eqref{e.3.27c}.
\end{proof}

\begin{prop}
	Let Assumption \ref{a:1} be satisfied and assume that $V(x)<\infty$ for $x\in \Sb $.
	The value function
	$V$ is a viscosity supersolution of the system of quasi-variational inequalities
	\eqref{e.3.17}; that is,
	 for
	 any $x^0\in S$ and any function $\ph\in C^2(S)$ satisfying
	 \beq{e.3.27j}(V-\ph)(x)\le (V-\ph)(x^0)=0,\eeq for all $x$ in a neighborhood of $x^0$, we have
	 \beq{e.3.27k}\max\limits_{i} \bigg\{(\L-\delta)\ph(x^0), f_i(x^0)-\dfrac{\partial \ph}{\partial x_i}(x^0),  \dfrac{\partial \ph}{\partial x_i}(x^0)-g_i(x^0)\bigg\}\ge 0.\eeq
\end{prop}

\begin{proof}
	Let $x^0\in S$ and suppose $\varphi(\cdot)\in C^2(S)$ satisfies \eqref{e.3.27j}
for all $x$ in a neighborhood of $x^0$.

We argue by contradiction. Suppose that \eqref{e.3.27k} does not hold. Then there exists a constant $A>0$ such that
	\beq{e.3.27m}\max\limits_{i} \bigg\{(\L-\delta)\varphi(x^0), f_i(x^0)-\dfrac{\partial \varphi}{\partial x_i}(x^0),  \dfrac{\partial \varphi}{\partial x_i}(x^0)-g_i(x^0)\bigg\}\le -2A< 0.\eeq
Let $\e>0$ be small enough so that $\overline{B_\e(x^0)}\subset S$ and for any $x\in \overline{B_\e (x^0  )}$, $\varphi(x)\ge V(x)$ and
	\beq{e.3.27n}
	\max\limits_{i} \bigg\{(\L-\delta)\varphi(x), f_i(x)-\dfrac{\partial \varphi}{\partial x_i}(x),  \dfrac{\partial \varphi}{\partial x_i}(x)-g_i(x)\bigg\}\le -A< 0.
	\eeq
		Let
	 $(Y, Z)\in \mathcal{A}_{x^0}$ and ${X}$ be the corresponding harvested process. Recall that $\theta=\inf\{t\ge 0: {X}(t)\notin B_\e(x^0) \}$ and $\theta_r=\theta\wedge r$ for any $r>0$.
 It follows from the Dynkin formula that
	\beq{e.3.27p}\barray \ad \E e^{-\delta \theta_r} \varphi({X}(\theta_r-)-\varphi(x^0)) \\
	\ad \quad = \E \int_0^{\theta_r -} e^{-\delta s} (\L -\delta ) \varphi ({X}(s))ds - \E \int_0^{\theta_r-} e^{-\delta s}  \nabla \varphi ({X}(s)) \cdot dY^c (s)\\
	\ad \quad + \E\int_0^{\theta_r-} e^{-\delta s}  \nabla \varphi ({X}(s)) \cdot dZ^c (s) +  \E \sum\limits_{0\le s< \theta_r} e^{-\delta s} \Big[ \varphi({X}(s) )- \varphi({X}(s-)) \Big]
.	\earray
	\eeq
	By virtue of Lemma \ref{lem}, for any $s\in [0, \theta_r)$,
	there exist  $\widehat{X}(s)\in \rr^d$ and $\widetilde{X}(s)\in \rr^d$ such that $\widehat{X}(s)\le X(s)$, $\widetilde{X}(s)\le X(s)$, and $$\varphi({X}(s) )- \varphi({X}(s-))\le -\Delta Y(s)\cdot \nabla \varphi (\widehat{X}(s)) + \Delta Z(s)\cdot \nabla \varphi (\widetilde{X}(s)).$$
This, together with the monotonicity of the functions $f$, $g$, and equation \eqref{e.3.27n} imply that
	\bea\varphi({X}(s) )- \varphi({X}(s-))\ad \le ( -f({X}(s)) -A\one )\cdot \Delta Y(s)\\
	\ad \qquad + ( g({X}(s)) -A \one )\cdot \Delta Z(s).\eea
	Hence it follows from \eqref{e.3.27n} and \eqref{e.3.27p} that
	\beq{}\barray \ad \E e^{-\delta \theta_r} \varphi({X}(\theta_r-))-\varphi(x^0)) \\
	\ad \quad \le  \E \int_0^{\theta_r -} e^{-\delta s} (-A)ds + \E \int_0^{\theta_r-} e^{-\delta s} (-f ({X}(s))-A\one) \cdot dY^c (s)\\
	\ad \quad + \E\int_0^{\theta_r-} e^{-\delta s}  (g({X}(s)) -A\one ) \cdot dZ^c (s) +  \E \sum\limits_{0\le s< \theta_r} e^{-\delta s} ( -f({X}(s)) -A\one )\cdot \Delta Y(s)\\
	\ad \qquad  + \E \sum\limits_{0\le s< \theta_r} e^{-\delta s} ( g({X}(s)) -A\one  )\cdot \Delta Z(s)\\
	\ad \quad =- \E \int_0^{\theta_r -} e^{-\delta s} f ({X}(s))\cdot dY(s) + \E \int_0^{\theta_r -} e^{-\delta s} g ({X}(s))\cdot dZ(s)-A \E \int_0^{\theta_r -} e^{-\delta s} ds \\
	\ad 	\qquad - A \E \int_0^{\theta_r -} e^{-\delta s}\one  \cdot dY(s) - A \E \int_0^{\theta_r -} e^{-\delta s} \one \cdot dZ(s).
	\earray
	\eeq
	Therefore,
	\beq{e.3.28}
	\barray
	\varphi(x^0) \ad \ge \E e^{-\delta \theta_r} \varphi ({X}(\theta_r-))  + \E \int_0^{\theta_r -} e^{-\delta s} f ({X}(s-))\cdot dY(s) \\
	\ad \quad - \E \int_0^{\theta_r -} e^{-\delta s} g ({X}(s-))\cdot dZ(s)\\
	\ad \quad + A \E\bigg[ \int_0^{\theta_r} e^{-\delta s} ds  +  \int_0^{\theta_r -} e^{-\delta s}\one \cdot  dY(s) +  \int_0^{\theta_r -} e^{-\delta s}\one \cdot  dZ(s) \bigg].
	\earray
	\eeq	
	We are in a position to apply Lemma \ref{lem:1}. To this end, recall from Lemma \ref{lem:1} that $\lambda$ is a random variable such that if $X(\theta)=X(\theta-)$, then $\lambda=0$; and if $X(\theta-)\ne X(\theta)$, then
	$\lambda$ is the positive number in $(0,1]$
such that $${X}_\lambda={X}(\theta-) +\lambda ({X}(\theta) - {X}(\theta-)  )={X}(\theta-)-\lambda (\Delta Y(\theta )- \Delta Z(\theta ))\in \partial {B_\e(x^0)}.$$ Note that $\lambda$ is independent of $r$.
Also recall that
$${X}_\lambda^r={X}(\theta_r-) +\lambda ({X}(\theta_r) - {X}(\theta_r-)  )={X}(\theta_r-)-\lambda (\Delta Y(\theta_r )- \Delta Z(\theta_r ))\in  \overline{B_\e(x^0)}.$$
	Using the same argument as the one in Lemma \ref{lem}, we obtain
	\beq{e.3.29}\barray	\varphi ( {X}(\theta_r-) ) - \varphi ( {X}_\lambda^r)\ad \ge \lambda [f({X}(\theta_r-)) +A \one]\cdot \Delta Y(\theta_r)  \\
	\ad \quad + \lambda [A\one - g({X}(\theta_r-))]\cdot \Delta Z(\theta_r).\earray\eeq
	Combining \eqref{e.3.28} and \eqref{e.3.29}, we have
	\beq{e.3.30}
	\barray
\ad 	V(x^0) = \varphi(x^0) \\
	\ad \ge \E e^{-\delta s} \varphi ( {X}_\lambda^r)  + \E \int_0^{\theta_r -} e^{-\delta s} f ({X}(s-))\cdot dY(s)\\
	\ad \quad
	- \E \int_0^{\theta_r -} e^{-\delta s} g ({X}(s-))\cdot dZ(s)\\
	\ad \quad  + A  \E\bigg[ \int_0^{\theta_r} e^{-\delta s} ds  +  \int_0^{\theta_r -} e^{-\delta s}\one\cdot  dY(s) +  \int_0^{\theta_r -} e^{-\delta s}\one\cdot  dZ(s) \bigg]  \\
	\ad \quad +  \lambda \E e^{-\delta \theta_r} [f({X}(\theta_r-)) +A\one ]\cdot \Delta Y(\theta_r)  + \lambda \E e^{-\delta \theta_r} [A\one - g({X}(\theta_r-))]\cdot \Delta Z(\theta_r).
	\earray
	\eeq
	Since ${X}_\lambda^r\in \overline{B_\e(x^0)}$, $\varphi ({X}_\lambda^r)\ge V({X}_\lambda^r)$. On the other hand, it follows from \eqref{e.3.5} that
	\beq{e.3.31}
	\barray
	V({X}_\lambda^r)\ad \ge V({X}(\theta_r))\\
	\ad \quad  + (1-\lambda)\Delta Y(\theta_r)\cdot f({X}_\lambda^r) -(1-\lambda) \Delta Z(\theta_r)\cdot  g( {X}_\lambda^r)\\
	\ad \ge \quad V({X}(\theta_r))\\
	\ad \quad + (1-\lambda)\Delta Y(\theta_r)\cdot f({X}(\theta_r-)) -(1-\lambda) \Delta Z(\theta_r)\cdot  g( {X}(\theta_r-)).
	\earray
	\eeq
	By \eqref{e.3.31} and \eqref{e.3.30} we note that
	\beq{e.3.31a}
	\barray
	V(x^0) \ad \ge \E \int_0^{\theta_r -} e^{-\delta s} f ({X}(s-))\cdot dY(s)
	- \E \int_0^{\theta_r -} e^{-\delta s} g ({X}(s-))\cdot dZ(s)\\
	\aad \quad + A \E\bigg[ \int_0^{\theta_r} e^{-\delta s} ds  +  \int_0^{\theta_r -} e^{-\delta s}\one \cdot  dY(s) +  \int_0^{\theta_r -} e^{-\delta s}\one\cdot  dZ(s) \bigg] + \E e^{-\delta \theta_r} 	V({X}(\theta_r))\\
	\aad\quad  + \lambda \E e^{-\delta \theta_r} [f({X}(\theta_r-)) +A \one ]\cdot \Delta Y(\theta_r)  + \lambda \ e^{-\delta \theta_r} [A\one - g({X}(\theta_r-))]\cdot \Delta Z(\theta_r)\\
	\aad\quad + (1-\lambda)\E e^{-\delta \theta_r} \Delta Y(\theta_r)\cdot f({X}(\theta_r-)) -(1-\lambda) \E e^{-\delta \theta_r} \Delta Z(\theta_r)\cdot  g( {X}(\theta_r))\\
	\ad  \ge  \E e^{-\delta \theta_r} 	V({X}(\theta_r))+\E \int_0^{\theta_r} e^{-\delta s} f ({X}(s-)\cdot dY(s)
	- \E \int_0^{\theta_r} e^{-\delta s} g ({X}(s-))\cdot dZ(s)\\
	\aad \quad + A \E\Big[ \int_0^{\theta_r} e^{-\delta s} ds  + \lambda e^{-\delta \theta_r} \one \cdot\Delta Z(\theta_r)  + \lambda e^{-\delta \theta_r} \one \cdot\Delta Y(\theta_r)\\
	\aad\quad
%\qquad\qquad \qquad \qquad\qquad
+ \int_0^{\theta_r -} e^{-\delta s}\one \cdot  dY(s) +  \int_0^{\theta_r -} e^{-\delta s}\one\cdot  dZ(s) \Big].
	\earray
	\eeq
Letting $r\to \infty$, we have
	\beq{e.3.31b}
	\barray
	V(x^0) \ad   \ge \E e^{-\delta \theta} 	V({X}(\theta))+ \E \int_0^{\theta} e^{-\delta s} f ({X}(s-)\cdot dY(s)
	- \E \int_0^{\theta} e^{-\delta s} g ({X}(s-))\cdot dZ(s)\\
	\aad \quad + A \E\Big[ \int_0^{\theta} e^{-\delta s} ds  + \lambda e^{-\delta \theta} \one \cdot\Delta Z(\theta)  + \lambda e^{-\delta \theta} \one \cdot\Delta Y(\theta) \\
	\aad
%\qquad\qquad \qquad \qquad\qquad
   \quad + \int_0^{\theta -} e^{-\delta s}\one \cdot  dY(s) +  \int_0^{\theta -} e^{-\delta s}\one\cdot  dZ(s)\Big].
	\earray
	\eeq
Using \eqref{e.3.27c} and \eqref{e.3.31b}, we arrive at
	\beq{e.3.31c}\barray
	V(x^0)\ad  \ge  \E \int_0^{\theta} e^{-\delta s} f ({X}(s-))\cdot dY(s)
	- \E \int_0^{\theta} e^{-\delta s} g ({X}(s-))\cdot dZ(s)\\
	\ad \quad +  \E e^{-\delta \theta} 	V({X}(\theta)) +A\ka_0.
	\earray	\eeq
	Taking the supremum over $(Y, Z)\in \mathcal{A}_{x^0}$, it follows that
	\beq{e.3.31d}\barray
	V(x^0)\ad \ge \sup\limits_{\mathcal{A}_{x^0}} \E\bigg[ \int_0^{\theta} e^{-\delta s} f ({X}(s-))\cdot dY(s)
	-\int_0^{\theta} e^{-\delta s} g ({X}(s-))\cdot dZ(s) \\
	\ad \quad +  \E e^{-\delta \theta} 	V({X}(\theta)) +A\ka_0\bigg].
	\earray
	\eeq
	In view of the dynamic programming principle, \eqref{e.3.31d} can be rewritten as
	$$V(x^0)\ge V(x^0) +A\ka_0 > V(x^0),$$ which is a contradiction. As a result \eqref{e.3.27k} has to hold, i.e. $V$ is viscosity supersolution of \eqref{e.3.17}.
\end{proof}

Summarizing what have obtained thus far, we state the following result.

\begin{thm} \label{thm:vis}
	Let Assumption \ref{a:1} be satisfied and suppose $V(x)<\infty$ for $x\in {S}$.
	The value function
	$V$ is a viscosity subsolution and also a viscosity supersolution, and hence
a viscosity solution, of the system of quasi-variational inequalities \eqref{e.3.17}.
	\end{thm}

\section{Numerical Algorithm}
\label{sec:alg}

\subsection{Transition Probabilities and Local Consistency}
We use the notation defined in Section \ref{sec:num}. To proceed, we state
one more assumption
below, which  will be used to ensure the validity of transition probabilities $p^h(x, y|\pi)$. However, it is not an essential assumption. There are several alternatives to handle the cases when Assumption \ref{a:2} fails. We refer the reader to \cite[page 1013]{Kushner90} for a detailed discussion. Define for any $x\in \bar S$ the covariance matrix $a(x)= \sg(x)\sg'(x)$.
\begin{asm}\label{a:2}
	For any $i=1, \dots, d$ and $x\in \Sb $, $$a_{ii}(x)-\sum\limits_{j: j\ne i}\big|a_{ij}(x)\big|\ge 0.$$
\end{asm}
Let $\E^{h, \pi}_{x, n}$, $\Cov^{h, \pi}_{x, n}$ denote the conditional expectation and covariance given by
$$\{X_m^h, \pi_m^h, m\le n, X_n^h=x, \pi^h_n=\pi \},$$
respectively. Our objective in this subsection is to define transition probabilities $p^h (x, y | \pi)$ so that the controlled Markov chain $\{X^h_n\}$ is locally consistent with respect to the diffusion \eqref{e.2}
in the sense that the following conditions hold:
\beq{e.4.5a}
\barray
\aad \E^{h, 0}_{x, n}\Delta X_n^h = {b}(x) \Delta t^h(x) + o(\Delta t^h(x)),\\
\aad Cov^{h, 0}_{x, n}\Delta X_n^h = a(x)\Delta t^h(x) + o(\Delta t^h(x)),\\
\aad \sup\limits_{n, \ \omega} |\Delta X_n^h| \to 0 \ \hbox{ as } \ h \to 0.
\earray
\eeq
To this end, using the procedure in \cite{Kushner90},
we define the approximation to the first and the second derivatives of $V$ by a finite difference method
using the step size $h>0$ for the state variable.
Afterwards,
 we plug in all the approximations into the first part of system \eqref{e.3.17}, combine similar terms and divide by the coefficient of $V^h(x)$. The transition probabilities are the coefficients of the resulting equation.
For $x\in S_h$, define
\beq{e.4.7}
\barray
\aad Q_h (x)=\sum\limits_{i=1}^d a_{ii}(x) -\sum\limits_{i, j: i\ne j}\dfrac{1}{2}|a_{ij}(x)| +h\sum\limits_{i=1}^d |b_i(x)|,\\
\aad \Delta t^h(x)=\dfrac{h^2}{Q_h(x)},\\
\aad p^h \(x, x+h\ei |\pi=0\) =
\dfrac{a_{ii}(x)/2-\sum\limits_{j: j\ne i}|a_{ij}(x)|/2+b_{i}^+(x) h }{Q_h (x)}, \\
\aad p^h \(x, x-h \ei) | \pi=0\) =
\dfrac{a^{ii}(x)/2-\sum\limits_{j: j\ne i}|a^{ij}(x)|/2+b_{i}^-(x) h}{Q_h (x)}, \\
\aad p^h \( x, x+h\ei +h \ej | \pi=0\) =  p^h \( x, x-h\ei -h\ej | \pi=0\) =
\dfrac{a_{{ij}^+}(x)}{2Q_h (x)}, \\
\aad p^h \( x, x+h\ei -h \ej | \pi=0\) = p^h \( x, x-h\ei + h \ej |  \pi=0 \)=
\dfrac{a_{{ij}^-}(x)}{2Q_h (x)},
\earray
\eeq
where for a real number $c$,  $c^+=\max\{c, 0\}$,
$c^-=-\min\{0, c\}$; that is, $c=c^+$ if $c\ge 0$ and $c=-c^-$ if $c<0$.
Set $p^h \(x, y|\pi=0\)=0$ for all unlisted values of $y\in S^h$.
Assumption \ref{a:2} guarantees that
the transition probabilities in \eqref{e.4.7} are well-defined. At the seeding and harvesting steps, we define
\beq{e.4.8}
\barray
\aad p^h\( x, x - h \ei| \pi = i\)=1,\\
\aad
p^h\( x, x + h \ei| \pi =- i\)=1.
\earray
\eeq
Thus, $p^h \(x,  y|\pi=\pm i\)=0$ for all nonlisted values of $y\in S^h$.
Using the above transition probabilities,
we can check
%that
 whether the locally consistent conditions of $\{X^h_n\}$ in \eqref{e.4.5a} are satisfied.

\subsection{Continuous--time interpolation and time rescaling}

The convergence
 result is based on a continuous-time interpolation of the chain,
which will be  constructed to be piecewise constant on the time interval $[t^h_n, t^h_{n+1}), n\ge 0$.
For use in this construction, we define
$n^h(t)=\max\{n: t^h_n\le t\}, t\ge 0$.
We first define discrete time processes associated with the controlled Markov chain as follows. Let
$Y^h_0=Z^h_0=B^h_0=M^h_0=0$ and define for $n\ge 1$,
\beq{e.4.11}
\barray
\aad Y^h_n = \sum\limits_{m=0}^{n-1}\Delta Y^h_m, \quad Z^h_n = \sum\limits_{m=0}^{n-1}\Delta Z^h_m,\\
\aad B^h_n = \sum\limits_{m=0}^{n-1} I_{\{\pi_m^h=0\}}
\E^{h}_m \Delta\xh_m, \quad M^h_n  = \sum\limits_{m=0}^{n-1} (\Delta \xh_m -
\E^{h}_m \Delta X_m)I_{\{\pi_m^h=0\}}.
\earray
\eeq
The piecewise constant interpolations, denoted by $(X^h\cd, Y^h\cd, Z^h\cd, B^h\cd, M^h\cd)$ are naturally defined as
\beq{e.4.12}
\barray
\aad X^h(t) = X^h_{n^h(t)},\\
\aad Y^h(t) = Y^h_{n^h(t)}, \quad Z^h(t) = Z^h_{n^h(t)},\\
\aad B^h(t) = B^h_{n^h(t)}, \quad M^h(t) = M^h_{n^h(t)}, \quad t\ge 0.
\earray
\eeq
Define $\mathcal{F}^h(t)=\sigma\{X^h(s), Y^h(s), Z^h(s): s\le t\}=\mathcal{F}^h_{n^h(t)}$.
Using the representation of diffusion, harvesting, and seeding steps in \eqref{e.4.1}, we obtain
\beq{e.4.13}
\barray
X^h_n = x + \sum\limits_{m=0}^{n-1} \Delta X_m^h \one_{\{ \pi^h_m\le -1\}} +  \sum\limits_{m=0}^{n-1} \Delta X_m^h \one_{\{ \pi^h_m\ge 1\}}+ \sum\limits_{m=0}^{n-1} \Delta X_m^h \one_{\{ \pi^h_m=0\}}
\earray
\eeq
This implies
\beq{e.4.14}
X^h(t)
= x + B^h(t) + M^h(t) -Y^h(t)+Z^h(t).
\eeq
Recall that $\Delta t^h_m = h^2/Q_h(X^h_m)$ if $\pi^h_m=0$ and $\Delta t^h_m = 0$ if $\pi^h_m\ge 1$ or $\pi^h_m\le -1$. It follows that
\beq{e.4.15}
\barray
B^h(t) \ad = \sum\limits_{m=0}^{n^h(t)-1} b (X^h_m)\Delta t^h_m\\
\ad=\int_0^t b (X^h(s)) ds-\int_{t^h_{n^h(t)}}^t b (X^h(s)) ds\\
\ad = \int_0^t b (X^h(s)) ds + \e^h_1(t),
\earray
\eeq
with $\{\e_1^h\cd\}$
%is
 be ing an $\mathcal{F}^h(t)$-adapted process satisfying \begin{equation*}\lim\limits_{h\to 0} \sup\limits_{t\in [0, T_0]}\E|\e_1^h(t)|=0 \quad \text{for any }0<T_0<\infty.\end{equation*} We now attempt to represent $M^h\cd$ in a form similar to the diffusion term in \eqref{e.2}.
Factor
$$a(x)= \sg(x)\sg'(x)=P(x)D^2(x)P'(x),$$
where $P\cd$ is an orthogonal matrix, $D\cd=\diag\{r_1\cd, ..., r_d
\cd\}$.
Without loss of generality, we suppose that
$\inf\limits_{x}r_i(x)>0$ for all $i=1, ..., d$. Define $D_0\cd=\diag\{1/r_1\cd, ..., 1/r_d
\cd\}$.  
\begin{rem} In the argument above, for simplicity, we assume that the diffusion matrix $(a(x))$ is nondegenerate. If this is not the case, we can use the trick from \cite[p.288-289]{Kushner92} to establish equation  \eqref{e.4.17}. 	
\end{rem}
Define $W^h\cd$ by
\beq{e.4.16}
\barray
W^h(t)  \ad = \int_0^t D_0 (X^h(s))
P' (X^h(s))dM^h(s)\\
\ad = \sum\limits_{m=0}^{n^h(t)-1} D_0 (X^h_m)
P' (X^h_m)(\Delta \xh_m -\E^{h}_m \Delta\xh_m)I_{\{ \pi^h_m=0\}}.
\earray
\eeq
Then we can write
\beq{e.4.17}
M^h(t) =\int_0^t \sg (\xh(s)) dW^h(s) + \e_2^h(t),
\eeq
with $\{\e_2^h\cd\}$
%is
 being an $\mathcal{F}^h(t)$-adapted process satisfying \begin{equation*}\lim\limits_{h\to 0} \sup\limits_{t\in [0, T_0]}\E|\e_2^h(t)|=0 \quad \text{for any }0<T_0<\infty.\end{equation*}
Using  \eqref{e.4.15} and \eqref{e.4.17}, we can write \eqref{e.4.14} as
\beq{e.4.18}
X^h(t) = x + \int_0^t b (X^h(s)) ds +  \int_0^t \sg(X^h(s)) dW^h(s) -Y^h(t)+Z^h(t)+\e^h(t),
\eeq
where $\e^h\cd$ is an $\mathcal{F}^h(t)$-adapted process satisfying \begin{equation*}\lim\limits_{h\to 0} \sup\limits_{t\in [0, T_0]}\E|\e^h(t)|=0 \quad \text{for any }0<T_0<\infty.\end{equation*}
The objective
function from \eqref{e.4.4} can be rewritten as
\beq{e.4.19}
J^h(x, Y^h, Z^h)= \E \int_0^{\infty} e^{-\delta s}\Big[ f(X^h(s-)) \cdot dY^{h}(s)-g(X^h(s-))\cdot  dZ^{h}(s)\Big].
\eeq

\noindent{\bf Time rescaling.} Next we will introduce ``stretched-out'' time scale. This is similar to the approach previously used by Kushner \cite{Kushner91} and Budhiraja and Ross \cite{Budhiraja07} for singular control problems. Using
the new time scale, we can overcome the possible non-tightness of the family of processes $\{Y^h\cd, Z^h\cd\}_{h>0}$.

Define the rescaled
time increments $\{\Delta \wdh{t}_n^h: n=0, 1, ...\}$ by
\beq{e.4.20}
\barray
\aad \Delta \wdh{t}^h_n = \Delta t^h_n I_{\{ \pi^h_n =0 \}} + h I_{\{ \pi^h_n \le -1 \}} + h I_{\{ \pi^h_n \ge 1 \}},\\
\aad \wdh{t}_0=0, \qquad\wdh{t}_n = \sum\limits_{k=0}^{n-1}\Delta \wdh{t}^h_k, \quad n\ge 1.\\
\earray
\eeq
The time scale is stretched out by $h$ at the seeding and harvesting steps.

\begin{defn}\label{def:1}{\rm
		The rescaled time process
		$\wdh{T}^h\cd$ is the unique continuous nondecreasing process satisfying the following:
		\begin{itemize}
			\item[(a)] $\wdh{T}^h(0)=0$;
			\item[(b)] the derivative of $\wdh{T}^h\cd$ is 1 on $(\wdh{t}^h_n, \wdh{t}^h_{n+1})$ if $\pi^h_n=0$, i.e., $n$ is a diffusion step;
			\item[(c)] the derivative of $\wdh{T}^h\cd$ is 0 on $(\wdh{t}^h_n, \wdh{t}^h_{n+1})$ if $\pi^h_n\ne 0$, i.e., $n$ is a seeding step or a harvesting step.	
		\end{itemize}	
	}
\end{defn}

Thus $\wdh{T}^h\cd$ does not increase at the times $t$ at which a harvesting step or a seeding step occurs.
Define the rescaled and interpolated process $\wdh{X}^h(t)= \xh(\wdh{T}^h(t))$ and likewise define $\wdh{Y}^h\cd$, $\wdh{Z}^h\cd$, $\wdh{B}^h\cd$, $\wdh{M}^h\cd$, and the filtration $\wdh{\mathcal{F}}^h\cd$ similarly.
It follows from \eqref{e.4.14} that
\beq{e.4.21}
\wdh{X}^h(t)=x+\wdh{B}^h(t) + \wdh{M}^h(t) - \wdh{Y}^h(t) + \wdh{Z}^h(t).
\eeq
Using the same argument we used for \eqref{e.4.18}
we obtain \beq{e.4.22}
\wdh{X}^h(t) = x + \int_0^t b (\wdh{X}^h(s))d\wdh{T}^h(s) + \int_0^t \sg(\wdh{X}^h(s)) d \wdh{W}^h(s)  - \wdh{Y}^h(t) + \wdh{Z}^h(t) + \wdh{\e}^h(t),
\eeq
with $\wdh{\e}^h\cd$ is an $\wdh{\mathcal{F}}^h\cd$-adapted process satisfying \beq{e.4.23}\lim\limits_{h\to 0} \sup\limits_{t\in [0, T_0]}\E|\wdh{\e}^h(t)|=0 \quad \text{for any }0<T_0<\infty.\eeq

\subsection{Convergence}

Using weak convergence methods, we can obtain the convergence of the algorithms. The proofs to the following results are essentially the same as those in \cite{Jin12, Ky16} and we therefore omit the details.

\begin{thm}\label{thm:thm} Suppose Assumptions \ref{a:1} and \ref{a:2} hold.
	Let the approximating chain $\{X^h_n \}$ be constructed with transition probabilities defined in \eqref{e.4.7}-\eqref{e.4.8},
	$\big(X^h\cd, W^h\cd, Y^h\cd, Z^h\cd\big)$ be the continuous-time interpolation defined in \eqref{e.4.11}-\eqref{e.4.12}, \eqref{e.4.16}, and $\wdh{T}^h\cd$ be the process from
Definition {\rm\ref{def:1}}.
Let $\wdh{X}^h\cd, \wdh{W}^h\cd, \wdh{Y}^h\cd, \wdh{Z}^h\cd$ be the corresponding rescaled processes and denote
$$\wdh{H}^h\cd=\Big(\wdh{X}^h\cd, \wdh{W}^h\cd, \wdh{Y}^h\cd, \wdh{Z}^h\cd, \wdh{T}^h\cd\Big).$$
Then the family of processes
	$(\wdh{H}^h)_{h>0}$ is tight. 	
	As a result, $(\wdh{H}^h)_{h>0}$ has a weakly convergent subsequence with limit $$\wdh{H}\cd=\Big(\wdh{X}\cd, \wdh{W}\cd, \wdh{Y}\cd, \wdh{Z}\cd, \wdh{T}\cd\Big).$$
\end{thm}

%In what follows, for notational simplicity, we still denote the convergent subsequence of $\wdh{H}^h\cd$ by $\wdh{H}^h\cd$.
We proceed to characterize the limit process.

\begin{thm}\label{thm:thm4.4}  Suppose Assumptions \ref{a:1} and \ref{a:2} hold.
	Let $\wdh{\mathcal{F}}(t)$ be the $\sigma$-algebra generated by
	$$\{\wdh{X}(s), \wdh{W}(s), \wdh{Y}(s), \wdh{Z}(s), \wdh{T}(s):s \le t\}.$$
	Then the following assertions hold.
	\begin{itemize}
		\item[\rm (a)] $\wdh{W}(t)$ is an $\wdh{\mathcal{F}}(t)$-martingale with quadratic variation process $\wdh{T}(t)I_d$.
		\item[\rm (b)]
		$\wdh{Y}\cd$, $\wdh{Z}\cd$, and $\wdh{T}\cd$ are  nondecreasing and nonnegative.
		\item[\rm (c)] The limit processes satisfy
		\beq{e.5.1}
		\wdh{X}(t)=x +\int_0^t b(\wdh{X}(s))d\wdh{T}(s)+\int_0^t  \sg(\wdh{X}(s))d\wdh{W}(s)-\wdh{Y}(t)+\wdh{Z}(t), \quad t\ge 0.	 \eeq
	\end{itemize}
\end{thm}
For $t<\infty$, define the inverse $R(t)= \inf\{s: \wdh{T}(s)>t\}$. For any process $\wdh{\nu}\cd$, define the time-rescaled process $(\bar\nu(t))$ by $\bar \nu(t)= \wdh{\nu}(R(t))$. Let ${\mathcal{\bar F}}(t)$ be the $\sigma$-algebra generated by
	$\{{\bar X}(s), {\bar W}(s), {\bar Y}(s), {\bar Z}(s), {\bar R}(s): s\le t\}$. Let $V^h(x)$ and $V^U(x)$ be value the functions defined in \eqref{e.4.5} and \eqref{e:VU}, respectively.

\begin{thm}\label{thm:r}
	 Suppose Assumptions \ref{a:1} and \ref{a:2} hold.  The following assertions are true.
	\begin{itemize}
		\item[\rm(a)] $\bar R$ is right continuous, nondecreasing, and $\bar R(t)\to \infty$ as $t \to \infty$ with probability 1.
		\item[\rm(b)]  $\bar Y$ and $\bar Z$ are right-continuous, nondecreasing, nonnegative, and $\mathcal{\bar F}(t)$-adapted processes.
		\item[\rm(c)]
		$\bar W\cd$ is a standard $\mathcal{\bar F}(t)$ adapted $d$ dimensional Brownian motion,
		and \beq{e.5.2}
		{\bar X}(t)=x +\int_0^t b(\bar X(s))ds+\int_0^t  \sg(\bar X(s))d\bar W(s)-\bar Y(t)+\bar Z(t), \quad t\ge 0.
		\eeq
		\item[\rm(d)] For any $x\in [0, U]^d$, $V^h(x)\to V^U(x)$ as $h\to 0$.
	\end{itemize}
\end{thm}

\end{document}